\documentclass{article}

\usepackage{arxiv}

\usepackage[utf8]{inputenc}
\usepackage[T1]{fontenc}    
\usepackage{hyperref}  
\usepackage{url}     
\usepackage{booktabs}     
\usepackage{amsfonts}     
\usepackage{nicefrac}       
\usepackage{microtype}   
\usepackage{lipsum}
\usepackage{graphicx}

\usepackage{float}
\usepackage{tikz}
\usetikzlibrary{snakes}
\usetikzlibrary{positioning}
\usepackage{rotating}
\usepackage{multirow}

\usepackage{amsmath}
\usepackage{amsthm}
\usepackage{amssymb}
\usepackage{mathtools}
\usepackage{pgfplots}
\usepackage{enumerate}
\usepackage{algorithm2e}

\newtheorem{thm}{Theorem}
\newtheorem{prop}{Proposition}
\newtheorem{definition}{Definition}[section]

\newcommand{\icol}[1]{
\left(\begin{smallmatrix}#1\end{smallmatrix}\right)
}

\newcommand\smallO{
  \mathchoice
    {{\scriptstyle\mathcal{O}}}
    {{\scriptstyle\mathcal{O}}}
    {{\scriptscriptstyle\mathcal{O}}}
    {\scalebox{.7}{$\scriptscriptstyle\mathcal{O}$}}
  }

\title{Traveling salesman problem with time slots: Asymptotic analysis and resolution algorithm }

\author{
 Omar Rifki \\
  Univ. Littoral Cote d’Opale, \\LISIC, F-62100 Calais, France. \\
  \texttt{rifki@univ-littoral.fr} \\
   \And
Thierry Garaix \\
  Mines Saint-Etienne,\\ Univ Clermont Auvergne,\\ 
  CNRS, UMR 6158 LIMOS, \\CIS, Departement I4S, F-42023,\\
  Saint-Etienne, France.\\
  \texttt{garaix@emse.fr} \\
}

\begin{document}
\maketitle
\begin{abstract}
We develop an asymptotic approximation and bounds for the traveling salesman problem with time slots, {\em i.e.} when the time windows of points to visit are a partition of a given time horizon.
Although this problem is relevant in several delivery applications, operational research researchers did not pay close attention to it, contrary to the extensively studied general formulation with time windows.
Exploiting the specificity of this problem allows to solve more efficiently instances which may be hard to solve by specifically designed algorithms for the traveling salesman problem with time windows. 
The asymptotic analysis of the traveling salesman problem with time slots is a step toward developing new approximations for the general problem with time windows.
We discuss this case.
We equally provide a formulation of the asymptotic approximation under a worst-case demand distribution of the points to visit, based on the principle of the maximum entropy.
Computational results are given for the benchmarks of the literature for which the time slots are randomly generated.
\end{abstract}


\keywords{Asymptotic analysis \and Traveling salesman problem  \and Time windows \and Solving algorithm}

\section{Introduction}
One of the main issues in transport logistics is computing optimal Hamiltonian cycles of a number of points, an NP-hard problem famously known as the Traveling Salesman Problem (TSP).
When the number of points to visit increases, the practical solving becomes harder, even using heuristic approaches.
On the other hand, in a large number of situations, especially those involving decisions on the strategical and the tactical level about the design of routing problems, only knowing the optimal tour length is needed, not the  tour in itself. 
Having a continuous approximation, which is a closed-form formula, easily computed, of the optimal tour length can be highly beneficial in those cases.
The strategical and the tactical decisions could be made quickly without running any optimization.
For instance, the continuous approximations are used by postal services for districting and sizing territories, as in the case of the United States postal service, see \cite{rosenfield1992application, novaes2000continuous, lei2015dynamic}.
The sizing and the composition of the fleet is another tactical decision which could rely on the approximations of tour length, {\em e.g.} \cite{nourinejad2017continuous, franceschetti2017strategic}. 
In case of  the location routing problem, which is a difficult problem composed by two NP-hard problems: a facility location and a vehicle routing problem, the solution of the routing tour could be approximated by an asymptotic approximation, {\em e.g.} \cite{drexl2015survey}.
Continuous approximations can be also useful in situations where the location of the points to visit is not known in advance, as the formulas rely on probabilistic assumptions.

Logistic and freight distribution have often time-related constraints, which restrict the visit of each point in time to an interval. 
These time windows constraints can be imposed either by authorities, to limit the access for freight vehicles to city centers \cite{akyol2018determining} or to freight loading zones for instance, or by customers for delivery, commercial and transportation operations, or by patients in case of medical transport. 
Their application range is wide.
In general, time windows have a significant impact on reducing the efficiency of routes and increasing the traveled distances of vehicles \cite{figliozzi2007analysis}.
We are interested in a specific format of time windows: non-overlapping time windows which form a partition of a planning horizon. 
We term them time slots.
This structure becomes increasingly present in logistics planning due to the boom of e-commerce services and on-demand businesses, such as online grocery stores \cite{hungerlander2018solving}.
Actually, in order to increase customer satisfaction from one side, and the flexibility of delivery operations from the company side, these companies pre-arrange wide time slots for customers to choose from, such that the service is guaranteed for each customer in the chosen slot.
The union of the time slots covers a large interval along the day, which could offer the customer the ability to choose the most adequate slot.
Having the same temporal structure to construct day to day tours is also easier from the supplier side, which is not obliged to handle the cumbersome task of managing customers time windows that may differ in structure on a daily basis.  

The main goal of this paper is to extend the continuous asymptotic approximation of the TSP to the TSP with time slots (TSP-TS).
With time slot constraints, it is obvious that the maximum tour length is bounded by the maximum difference between the time windows. 
However, an approximation function make it possible to approach how the tour length evolves according to the number of customers and also to apprehend the feasibility criteria of the tour. In this context, the main difficulty is to simultaneously manage the geographical and the temporal distributions.
The contributions of this paper are listed as follows:
\begin{itemize}
    \item[-] Proposition of an asymptotic approximation, feasibility conditions, and asymptotic bounds for the TSP-TS in case of uniform temporal and spacial distributions;
    \item[-] Extension of this approximation to the worst-case temporal and spacial demands of customers;
    \item[-] Proposition of an exact solving approach for the TSP-TS;
    \item[-] Generation of a dataset benchmark for the TSP-TS.
\end{itemize}

The remaining of the article is organized as follows. 
Section 2 presents a brief literature review on the topic of continuous approximations in routing problems. 
Section 4 introduces the proposed asymptotic approximation and bounds accounting for the time windows consideration, while the preliminaries of the study are stated in Section 3. 
The worst-case demand in terms of space and time is treated in Section 5.
The solving approach is given in Section 6, while
Section 7 provides computational results. 
The paper is concluded thereafter. 

\section{Literature review}
The approximation of the routing problem is grounded on the famous theorem of Beardwood, Halton, and Hammersley (BHH), published in 1959 \cite{beardwood1959shortest}. 
This result states that when the number of points to visit is randomly distributed on a compact area and goes to infinity, the optimal tour length approaches a constant value. 
Noting that the BHH formula underestimates tour lengths in elongated areas even for a larger number of points, Daganzo \cite{daganzo1984length} proposed a strip strategy method, which computes optimized tour length in those types of areas.
The two models of \cite{beardwood1959shortest, daganzo1984length} gave rise to several extensions accounting for the variants of the transportation problems, and accommodating to the area's shape.
Chien \cite{chien1992operational} through a regression model accounted for the rectangular shape area in the TSP approximation by including the area of the smallest rectangle enclosing all points and the average distance to the depot.
Kwon \cite{kwon1995estimating} via regression as well and neural networks improved the TSP approximation by including the length to width ratio of the rectangle and a shape factor. 
The first approximations of the capacitated vehicle routing problem (CVRP) were proposed by Webb \cite{webb1968cost}.
Eilon {\em et al.} proposed a similar formula to the BHH for the CVRP accounting for the distances between the depot and the customers and for the shape of the support area. 
An intuitive approximation was equally proposed by Daganzo \cite{daganzo1984distance} for the CVRP.
BHH formula has also led to the development of solving heuristics such as `Partition' \cite{karp1977probabilistic}.
For a review of the overall extensions deriving from the continuous asymptotic approximation of the TSP, see \cite{franceschetti2017continuous} and \cite{ansari2018advancements}.

Continuous approximations of the routing with time-windows concern mainly the vehicle routing problem (VRP).
Daganzo \cite{daganzo1987modeling1, daganzo1987modeling2} has developed a model wherein the day is divided into time periods and customers into rectangles. 
Using a cluster-first route-second method, he obtained an approximation for the total distance traveled by all vehicles under these considerations.
Figliozzi \cite{figliozzi2009planning} tested several VRP approximations, and proposed a probabilistic modeling of the approximation such that the number of routes for a given number of time windows is derived probabilistically. 
Nicola {\em et al.} \cite{nicola2019total} proposed 
regression-based approximations for the TSP, the CVRP with time windows, the multi-region multi-depot pickup and delivery problem accounting for the time windows, distances, customer demands and capacities of the vehicles.

Using similar assumptions to \cite{daganzo1987modeling1, daganzo1987modeling2}, Carlson and Behzoodi \cite{carlsson2017worst} studied the worst-case time window distribution in terms of routing costs, and found that it corresponds to a concentrated demand on a single time period when the number of customers is low, or to a uniform distribution over the time for a large number of customers.
Although VRP asymptotic approximations are intuitive and simple to use, they are mainly grounded on empirical evidence as opposed to the analytical derivation of the BHH theorem.

There has been several recent applications, whether for districting \cite{lei2016solving}, location problems \cite{wang2017continuum}, fleet sizing \cite{franceschetti2017strategic} or accounting for pickups and deliveries \cite{bergmann2020integrating}.

Our model is based for the time windows considerations on similar assumptions to \cite{daganzo1987modeling1, daganzo1987modeling2, carlsson2017worst} in the sense of taking non-overlapping intervals as time windows. However, we differ from the previous accounts by pursuing a theoretical derivation of the asymptotic approximations from the BHH formula.   

\section{Preliminaries}
In this section, we present the formulation of the routing problem, the assumptions to generate random instances, and some previously obtained asymptotic approximations, starting from the famous BHH formula.

\subsection{The TSP with Time Windows (TSP-TW)}
The set of points to visit is denoted ${\cal P}$, and is a finite set $|{\cal P}|<\infty$. 
The depot is denoted by the point $0$, and ${\cal P}_{all}$ denotes the set of all points, {\em i.e.} ${\cal P}_{all} = {\cal P} \cup \{0\}$.
For each point $i\in {\cal P}_{all}$, we denote $b_i$, $f_i$, and $s_i$ the earliest starting time, the latest finishing time associated with $i$, and the service time respectively. 
For each couple of points $i, j\in {\cal P}_{all}$, $d_{ij}$ and $c_{ij}$ denote respectively the travel duration between $i$ and $j$, and the cost of traversing the arc $(i,j)$.
The specifications of the costs $c_{ij}$, and the times $d_{ij}$ are discussed in the assumption section.
The goal is to find an order of visit that minimizes the total tour duration of the vehicle starting from the depot.
For a couple of points $i, j \in {\cal P}_{all}$, let $x_{ij}$ be a binary variable equal to one if and only if $j$ is visited after $i$.
For all $i \in {\cal P}_{all}$, let $t_i$ be the start service time of $i$. Note that the vehicle traveling from $i$ to $j$ can wait in case of early arrival at $j$, \textit{i.e.}, $t_i+s_i<b_j$ is allowed. The parameter 
$t_0$ represents the start service time from the depot at the beginning of the tour.    
A formulation of the problem is as follows:

\begin{eqnarray}
\min \quad \sum_{i, j\in {\cal P}_{all}} c_{ij} x_{ij} &&\\
\sum_{j\in {\cal P}_{all}\setminus \{i\} } x_{ij} =1 && \forall i\in {\cal P}_{all}\\
\sum_{j\in {\cal P}_{all}\setminus \{i\} } x_{ji} =1 && \forall i\in {\cal P}_{all}\\
t_j \geq  t_i + s_i + d_{ij} + M (x_{ij}-1) && \forall i\in {\cal P}_{all}\;\; \forall j\in {\cal P}\\
f_0 \geq  t_i + s_i + d_{i0} && \forall i\in {\cal P}\\
b_i \leq t_i \leq f_i - s_i&& \forall i\in {\cal P}_{all} \\
t_{i}\geq 0 && \forall i \in {\cal P}_{all}\\
x_{ij}\in \{0,1\} && \forall i,j \in {\cal P}_{all}
\end{eqnarray}

Constraints $(1)$ and $(2)$ are the flow conservation constraints, ensuring that each point is visited exactly once.
Constraints $(3)$ and $(4)$ track the arrival times, with $M$ is a very large number that could take the value of $f_0$.
Constraint $(5)$ ensures that the arrival times satisfy the time window constraints.
Constraints $(6)$, and $(7)$ represent the binary and the bounding restrictions of the decision variables.

\subsection{The TSP with Time Slots (TSP-TS)}
The TSP with Time Slots is a special case of the TSP-TW. 
Time slots are defined as a partition of the time horizon $[0,h]$ of the problem.
Introducing a horizon $h$ in the formulation $(P_1)$ comes down to setting $f_0=h$.
Time slots represent non-overlapping time windows, with the characteristic that several points could be assigned to a same time slot.
Thus, new parameters are needed for the description: the number of time slots $m$, and the time slot lengths $(l_k)_{1\leq k \leq m}$.
We use the abbreviation TSP-$m$TS to refer to this problem.

The TSP with Identical Time Slots (TSP-ITS) is a variant of the TSP-TS, where time slots have the same length, {\em i.e.} $l_k=h/m$. 
We use the abbreviation TSP-$m$ITS for this special case.

TSP-TS provides a relevant model in several applications; delivery at home or technicians visits, which are offered in slots morning, afternoon or evening visits. 
TSP-TS definition is motivated in this paper by its asymptotic properties and its closeness to TSP-TW.

\subsection{Assumptions}
In this subsection, we discuss the assumptions on the characteristics of the input data of the addressed problem. 
 
\paragraph{Assumptions for TSP-TW.}
The problem input parameters that are used to randomly generate instances are: the size $n$ of ${\cal P}_{all}$, 
the time horizon $h$, 
and the side $a$ of the area where the points ${\cal P}_{all}$ are located. 
Given these parameters, the assumptions to generate a TSP-TW instance are:
\begin{enumerate}[(i)]
\item The points ${\cal P}_{all}$ are uniformly distributed on the {\em Euclidean} plane $\mathbb{R}^2$ within a square area $\mathcal{R}$, {\em i.e.}, $\mathcal{R} = \{(x,y) \in \mathbb{R}^2: 0\leq x\leq a, \;\;\text{and}\;\; 0\leq y\leq a \}$. 
The random variables associated with these $n$ points, $X_0, \ldots, X_{n-1}$, are supposed to be independent and identically distributed (i.i.d).
$X_0$ is the corresponding random variable of the depot.
\item The cost $c_{ij}$ is taken to be equal to the travel duration $d_{ij}$, and the stop durations are taken to be null $s_i=0\;\;(\forall i \in {\cal P}_{all})$. 
\item The duration $d_{ij}$ is the  {\em Euclidean} distance between $i$ and $j$ ({\em i.e.}, we assume that the vehicles have a constant speed of one unit of space per unit of time)\footnote{The given approximations are valid for any cost function at the condition to be proportional to the  {\em Euclidean} distance metric.}.
\item For optimal TSP tours in the asymptotic domain, we consider the return arc to the depot to be no different than any other arc of the tour in terms of travel duration.
\end{enumerate}

The last assumption is useful for constructing the asymptotic approximation using the BHH formula.
No specific assumption is associated with the time windows $[b_i, f_i]$, except for the depot: $b_0=0$ and $f_0=h$.

\paragraph{Assumptions for TSP-TS.}
The assumptions of the TSP-$m$TS are same as those of the TSP-TW, in addition to the following assumptions: 
\begin{enumerate}[(i)]
\setcounter{enumi}{4}
\item The time slots $A_k=[B_k,F_k]$, $k \in \{1, \ldots m\}$ are a partition of the time horizon $[0,h]$, with $B_1=0$, $F_m=h$ and $F_k=B_{k+1}\: \forall 1<k<m$. For the ease of exposition of the results we limit our study to the lengths $l_k = F_k - B_k$ which are integer dividers of $n$. Later, this choice allows to consider integer number of nodes to visit during each time slot. 
\item The points ${\cal P}$ are assigned to time slots $(A_k)$ at random in a proportional and an i.i.d. fashion. That the probability to assign a point to the time slot $A_k$ is $p_k = |A_k|/m = l_k/m$. This assumption is equivalent to draw uniformly at random a point in the time horizon $[0,h]$ and to assign this point to the corresponding time slot. 
\end{enumerate}
The assumptions of the TSP-$m$ITS are same as those of the TSP-$m$TS, in addition to the following:
\begin{enumerate}[(i)]
\setcounter{enumi}{6}
\item Each time slot $k \in \{1, \ldots m\}$ begins at the time $B_k=(k-1) \times h/m$ and ends at the time $F_k=k \times h/m$.
\end{enumerate}
The uniform distribution of the hypotheses (i), and (vi) are relaxed in Section 5, by considering realistic worst-case distributions.

\subsection{Approximations in the literature}
The main results obtained for the TSP, the VRP and the TSP-$m$ITS problems are discussed below. 

\paragraph{TSP asymptotic approximation} 
The Beardwood-Halton-Hammersley (BHH) theorem gives an approximation of the optimal tour length when the number of points to visit goes to infinity on a compact area \cite{beardwood1959shortest}. 
This length approaches a constant value. 
The theorem below is stated for a planar region given any probability spatial distribution of the points ${\cal P}_{all}$. 

\begin{thm}[\textbf{BHH}]
\label{thm1}
For a set of $n$ random variables $\{X_0, ..., X_{n-1}\}$ $(0<n<\infty)$ independently and identically distributed on a compact support $\mathcal{R} \subset \mathbb{R}^2$, the length $L_{n}^{tsp}$ 
of a shortest Hamiltonian path linking $X_i$ satisfies
$$\frac{L_{n}^{tsp}}{\sqrt{n}}  \xrightarrow[n\to \infty]{}  \beta^{tsp} \int_{\mathbb{R}^2}\sqrt{\overline{f}(x)}\;dx,$$
with $\overline{f}(.)$ is the absolutely continuous part of the probability density function $f(.)$ of the variables $X_i$, and $\beta^{tsp}$ is a constant. 
\end{thm}

Under the uniform probability distribution of $\{X_0, ..., X_{n}\}$, the BHH formula becomes,
$$\frac{L_{n}^{tsp}}{\sqrt{n}}  \xrightarrow[n\to \infty]{}  \beta^{tsp} \sqrt{|\mathcal{R}|}.$$
Thus, $L_{n}^{tsp} = \beta^{tsp} \sqrt{n\;|\mathcal{R}|}+\smallO(\sqrt{n}),$
where $|\mathcal{R}|$ is the surface of the planar area $\mathcal{R}$ where the points are distributed. 

\paragraph{Computation of $\beta^{tsp}$. }
The constant $\beta^{tsp}$ does not depend on the number of points $n$ in the infinite domain. 
However, for smaller values of $n$, several estimates are given for $\beta^{tsp}$ in the literature. 
According to Arlotto and Steele \cite{arlotto2016beardwood}, $\beta^{tsp}$ varies in the interval: $0.62499 \leq \beta^{tsp} \leq 0.91996$. Stein \cite{stein1978asymptotic} uses the estimate $\beta^{tsp} = 0.765$, while Applegate {\em et al.} \cite{applegate2006traveling} and Lei {\em et al.} \cite{lei2016solving} give practical estimates, which depend on the number $n$, as shown in Table~\ref{tab:beta}.

\begin{table}[h!]
    \centering
    \begin{tabular}{|c|c|}
        \hline
         $n$ & $\beta^{tsp}$ \\ \hline \hline
         20 & 0.8584265 \\ \hline
         30 & 0.8269698\\ \hline
         40 & 0.8129900\\ \hline
         50 & 0.7994125\\ \hline
         60 & 0.7908632\\ \hline
         70 & 0.7817751\\ \hline
         80 & 0.7775367\\ \hline
         90 & 0.7773827\\ \hline
    \end{tabular}
    $\;\;\;\;\;$
        \begin{tabular}{|c|c|}
        \hline
         $n$ & $\beta^{tsp}$ \\ \hline \hline
         100 & 0.7764689 \\ \hline
         200 & 0.7563542\\ \hline
         300 & 0.7477629\\ \hline
         400 & 0.7428444\\ \hline
         500 & 0.7394544\\ \hline
         600 & 0.7369409\\ \hline
         700 & 0.7349902\\ \hline
         800 & 0.7335751\\ \hline
    \end{tabular}
    $\;\;\;\;\;$
    \begin{tabular}{|c|c|}
        \hline
         $n$ & $\beta^{tsp}$ \\ \hline \hline
         900 & 0.7321114\\ \hline
         1000 & 0.7312235\\ \hline
         2000 & 0.7256264\\ \hline
    \end{tabular}
    \caption{Empirical estimates of $\beta^{tsp}$ given by  \cite{lei2016solving} for $n\leq 90$, and by \cite{applegate2006traveling} otherwise.}
    \label{tab:beta}
\end{table}

It is worth noting that the overall proposed $\beta^{tsp}$ values in the literature have been obtained by resolutions on squares surfaces, like those of Table~\ref{tab:beta}.
In theory $\beta^{tsp}$ does not depend on the shape of the support  surface or area. 
But, the impact of the shape can be of high significance.
As noted by \cite{daganzo1984length, chien1992operational, kwon1995estimating}, the BHH formula does not behave well when the base surface becomes elongated.
We show this fact in the following two figures by giving some computational results on rectangles of varying aspect ratios $\alpha$, {\em e.g.} the ratio of the width to the height of the rectangle.
The values of $\beta^{tsp}$ used for the TSP approximations in the figures are those of  Table~\ref{tab:beta} depending on the number $n$ of points.
Figure~\ref{fig:bp_vrp1} shows the gaps of the TSP approximation to the optimal tour length for several $n=1,000$ points instances uniformly distributed on rectangles with various $\alpha$ and the same surface. 
The tours are computed with the Concorde solver \cite{applegate1998solution}.
The TSP approximation becomes less accurate for higher values of the aspect ratio $\alpha$ on the same surface. 
Additionally, when the number of points to visit $n$ is small, the impact of $\alpha$ on the BHH formula is even greater.
Figure~\ref{fig:bp_vrp2} shows the quality gaps of the TSP approximation for several values of $n$.
For $n=10$ and $20$, the average of the absolute gaps is equal to resp. $8.6\%$ and $6\%$ on a square surface, and to $46.81\%$ and $38.12\%$ on a rectangle surface with $\alpha = 10$.
The BHH approximation hugely underestimates the optimal tour length for cases of large $\alpha$ and small $n$. When applying a TSP asymptotic approximation on an elongated surface, it could be useful to first derive a set of $\beta^{tsp}$ values, which better account for the shape of the working surface.

\begin{figure}[htb!]
    \centering
    \includegraphics[width=0.45\textwidth]{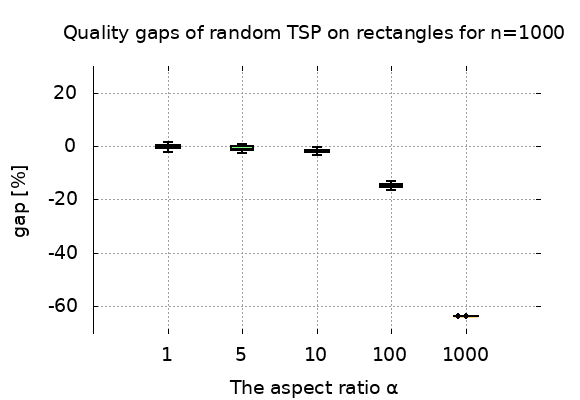}
    \caption{The distributions of the gaps of $L^{tsp}_{n}$ to the optimal tour lengths ($=(L^{tsp}_{n}/cost_n-1) \times 100$ where $cost_n$ is the optimal tour length) of instances of $n=1000$ points  uniformly distributed on rectangles with several aspect ratio $\alpha$. All the rectangles have the same surface area $10^6$. For each $\alpha\in \{1, 5, 10, 100, 1000\}$, 50 TSP instances are solved.}
    \label{fig:bp_vrp1}
\end{figure}

\begin{figure}[htb!]
    \centering
    \includegraphics[width=0.45\textwidth]{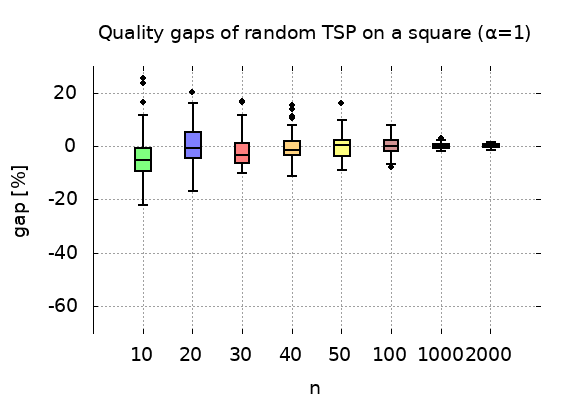}
    \includegraphics[width=0.45\textwidth]{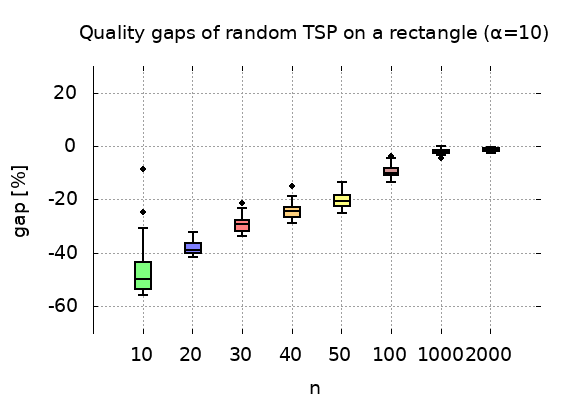}
    \caption{The distributions of the gaps of $L^{tsp}_{n}$  to the optimal tour lengths of instances of $n$ points uniformly distributed on rectangles with  $\alpha= 1$ (a square) and $\alpha = 10$. All the rectangles have the same surface area of $10^6$. For each $n\in \{10, 20, 30, 40, 50, 100, 1000, 2000\}$, 50 TSP instances are solved.}
    \label{fig:bp_vrp2}
\end{figure}






\paragraph{The TSP-$m$ITS.} 
Below a satisfiability condition and an asymptotic approximation for the TSP-$m$ITS case, previously presented in \cite{rifki2021asymptotic}.
The total number of points with the depot is equal to $n$.
We call an TSP-$m$ITS random generation model, a mechanism generating instances of the problem TSP-$m$ITS satisfying the related assumptions (see Section 3.1.2):\\

\noindent \textbf{Satisfiability condition of TSP-$m$ITS:}
For the TSP-$m$ITS random generation model, 
the satisfiability condition of feasible tours linking the realisations of $\{X_0, ..., X_{n-1}\}$ under the  {\em Euclidean} metric satisfies on average 
\begin{equation} \label{eq:pareto}
    n \times m=\frac{h^2}{|\mathcal{R}|\;(\beta^{tsp})^2}.
\end{equation}

\begin{prop}[\textbf{Approximation TSP-$m$ITS}]
\label{prop2}
For the TSP-$m$ITS random generation model,
the asymptotic length $L_{n}^{tsp-mits}$ of the shortest tour linking $\{X_0, ..., X_{n-1}\}$ under the  {\em Euclidean} metric if feasible is equal to 
\begin{align}
L_{n}^{tsp-mits} &= \beta^{tsp} \; \sqrt{n\;m\;|\mathcal{R}|}  + \smallO\Big(\sqrt{n\;m}\Big). \label{eq:tsp-mits}
\end{align}
\end{prop}

The idea behind equation (\ref{eq:pareto}) is to impose the condition that each time slot length is enough for visiting the points affected to this time slot. 
The approximation (\ref{eq:tsp-mits}) is obtained by considering that each time slot infers a TSP tour, such that the BHH theorem could be used. 
Both arguments will be revisited in Section 4.1 for the general case of the TSP-$m$TS.

\section{Asymptotic approximations and bounds}\label{sec:appr}
In this section, we provide some asymptotic approximations and bounds for the optimal tours of the TSP-$m$TS, and the TSP-TW.
Table~\ref{tab1} summarizes the obtained results.

\begin{table}[htb]
    \centering
    \begin{tabular}{|c||c|c|c|}
        \hline
         TW & satisf. & approx.& bound\\  
          constraints & & &\\ \hline
         $\emptyset$ & - & Theorem~\ref{thm1}&  \\\hline
         ITS & Eq(\ref{eq:pareto}) & Prop.~\ref{prop2} &\\ \hline
         TS & Eq(\ref{eq:pareto2}) & Prop.~\ref{prop4} & Prop.~\ref{prop5} \\ \hline
         TW &  & Prop.~\ref{prop6} & \\ \hline
         \end{tabular}
    \caption{Results obtained for the satisfiability condition (satisf.), the asymptotic approximation (approx.), and bounds (bound), for the TSP with different time windows constraints. 
    Equation Eq(\ref{eq:pareto}), Prop. 1, and Theorem 1 are stated in the preliminaries. 
    }
    \label{tab1}
\end{table}

\subsection{Asymptotic approximation of the TSP-$m$TS}
In the following a satisfiability condition and an asymptotic approximation for the TSP-$m$TS case.
We provide also a bound linking the approximation of TSP-$m$TS and that of TSP-$m$ITS.
The time slots $(A_k)_{1 \leq k\leq m}$ are considered inputs of the problem.\newline

\noindent \textbf{Satisfiability condition of TSP-$m$TS}:
For the TSP-$m$TS random generation model with the time slots $(A_k)_{1 \leq k\leq m}$, 
the satisfiability condition of feasible tours linking the realisations of $X_0, ...X_{n-1}$ under the  {\em Euclidean} metric satisfies on average
\begin{equation} \label{eq:pareto2}
    l_{min} =\frac{(\beta^{tsp})^2\;|\mathcal{R}|}{h}\times n,
\end{equation}
where $l_{min} = \min_{1\leq k \leq m} |A_k|$.\newline

In order to derive the condition (\ref{eq:pareto2}), we adopt the following approach.
As the distribution of points across time slots is proportional to their length, the expected number $E(n_k)$ of points visited in each time slot $A_k$ is equal to
$$ E(n_k) = \frac{n\times l_k}{h}.$$

To treat all time slots equally, we have considered $n$ points (instead of $n-1$) in this distribution of points. 
In the TSP and the TSP-TS context, the depot does not carry a special significance, except that it is the first point where the tour starts.
In terms of space, the depot is treated equally to the other points (assumption (i)).
In the asymptotic domain, one point difference in the set of points to visit has a negligible impact.

In order to have at least one feasible tour of the realizations of $X_i$, the duration of the shortest (Hamiltonian) path linking the points of each time slot and bridging to the surrounding slots must be at most equal to the size of the time slot.
Figure~\ref{fig0-satis} shows a representation of crossing the points of a time slot $A_i$.
The feasibility condition is given by $b_i+c_i+e_i \leq l_i$ for all $A_i$.
We approximate the travel time $b_i+c_i+e_i$ in the asymptotic domain by a BHH formula.
This formula allows us to approximate any path of $E(n_k)$ arcs  by a tour of $E(n_k)$ points since its formulation does not require knowing the specific positions of the points to visit.
For more precaution, we use assumption (iv), which allows us to treat the return arc of an optimal tour in the asymptotic domain like any other arc of the tour.
In this configuration, the proportion of $B_i$ and $C_i$ arcs in the time slot $A_i$ (see Figure~\ref{fig0-satis}) is treated as one arc.
Then, we suppose asymptotically that 
\begin{equation}\label{eq:pareto2-1}
L_{E(n_1)}^{tsp} \leq l_1,\;
L_{E(n_2)}^{tsp} \leq l_2,
...,\;
L_{E(n_m)}^{tsp} \leq l_m.
\end{equation}

\begin{figure}[ht]
\centering
\includegraphics[width=0.5\textwidth]{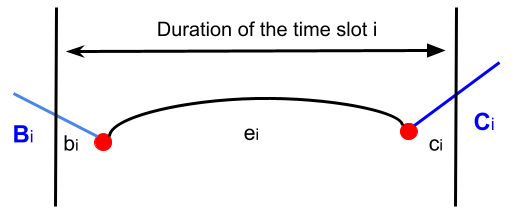}
\caption{A representation of crossing the points of a given time slot $A_i$. $B_i$ and $C_i$ represent the bridging arcs of $A_i$, while $b_i$ (resp. $c_i$) is the duration of crossing $B_i$ (resp. $C_i$) within the interval of the time slot. $e_i$ is the duration of the shortest path linking the points affected to the time slot. 
For the first time slot, $B_i$ does not exist and $b_i=0$.
For the last time slot, $C_i$ is the returning arc to the depot and $c_i=C_i$.}
\label{fig0-satis}
\end{figure}

\noindent Let us denote by $C= \beta^{tsp}\sqrt{|\mathcal{R}|/h}$, we have 
\begin{align*}
(\ref{eq:pareto2-1}) &\iff
C\times \sqrt{n} \leq \sqrt{l_k}\;\;\;\;\;\;\forall k \in \{1,2,...m\}\\
& \implies C\times \sqrt{n} \leq \sqrt{l_{min}}
\end{align*}

According to the BHH formula, the length of the tours of the time slots evolves as $\sqrt{n}$. 
The satisfiability of a TSP-TS tour of $n$ points is determined by the intersection between the function $C\times \sqrt{n}$, calibrated for the instance parameters ($|\mathcal{R}|$, $h$), and the value $\sqrt{l_{min}}$ of the smallest time slot length. 
Figure~\ref{fig1-satis} shows that for several $\sqrt{l_{min}}$ values the satisfiable number of customers varies: $n_1$ for $l_{min}=l_1$, and $n_2$ for $l_{min}=l_2$.
Thus, the upper bound of the number of customers $n$ that can be served asymptotically is limited by $l_{min}$, the horizon $h$, and the area $|\mathcal{R}|$, {\em i.e.}, 
$$n\leq n_{max}=\frac{l_{min}}{C^2} = \frac{l_{min} \times h}{(\beta^{tsp})^2\;|\mathcal{R}|}.$$

\begin{figure}[ht]
\centering
\begin{tikzpicture}
  \draw[->] (-0.5, 0) -- (5, 0) node[right] {$n$};
  \draw[->] (0, -0.5) -- (0, 2) node[above] {$y$};
  \draw[scale=0.1, domain=0:50, smooth, variable=\x, thick] 
  plot (\x, {1.5*sqrt(\x)}) node[above] {$C\times \sqrt{n}$};
  \draw[scale=0.1, domain=0:50, smooth, variable=\x, dashdotdotted]  
  plot (\x, 9) node [right] {$\sqrt{l_1}$};
  \draw[scale=0.1, domain=0:50, smooth, variable=\x, dashdotdotted]  
  plot (\x, 4.5) node[right] {$\sqrt{l_2}$};
  \draw[scale=1, domain=0:2, dashed, variable=\y, black]  
  plot (1, \y) node[above right] {$n_2$} ;
  \draw[scale=1, domain=0:2, dashed, variable=\y, black]  
  plot (4, \y) node[above right] {$n_1$} ;
\end{tikzpicture}
\caption{Plot of the function $y=C\sqrt{n}$, $y=\sqrt{l_1}$, and $y=\sqrt{l_2}$, with $l_1>l_2>0$.}
\label{fig1-satis}
\end{figure}
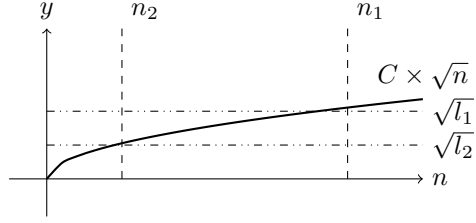

\noindent Similarly, if $n$ is fixed, the lower bound of the minimum length of the time slots  is given asymptotically by 
$$\frac{n\times(\beta^{tsp})^2\times|\mathcal{R}|}{h} \leq l_{min}.$$
The equality case, that is the equation (\ref{eq:pareto2}) defines a limiting curve of the satisfiable region of tours linking the realizations of $X_i$.

Since the distribution of points to visit on time windows follows an i.i.d uniform distribution, then for a given time slot $A_k$, the assignment of a point $i$ can be seen as a Bernoulli trial with a probability $p_k=l_k/h$. 
The random variable $n_k$ of the number of points of a time slot follows the binomial distribution $\mathcal{B}(n, p_k)$. Under the central limit theorem, $n_k$ converges asymptotically to a normal distribution.
Being a symmetric distribution, half of the realizations of $n_k$ are over $E(n_k)$, and the other half is below for each time slot $k$.
The equation (\ref{eq:pareto2}) is thus given on average following this argument.\\

\begin{prop}[\textbf{Approximation of TSP-$m$TS}]
\label{prop4}
For the TSP-$m$TS random generation model with the time slots $(A_k)_{1 \leq k\leq m}$,
the asymptotic length $L_{n}^{tsp-mts}$ of the shortest tour linking $X_0, ...X_{n-1}$ under the  {\em Euclidean} metric if feasible is equal to
\begin{equation}\label{eq:tsp-mts}
L_{n}^{tsp-mts}  = L_{n}^{tsp} \sum_{k=1}^{m} \sqrt{\frac{l_k}{h}} = \beta^{tsp} \; \sqrt{\frac{n\;|\mathcal{R}|}{h}} \sum_{k=1}^{m} \sqrt{l_k},
\end{equation}
where $l_k = |A_k|$, $\forall k \in \{1, \ldots m\}$.
\end{prop}

\begin{proof}
Since time slots are non-overlapping intervals, the optimal tour linking the realizations of $X_i$ in the asymptotic domain can be seen as a summation of optimal TSP tours, one for each time slot $A_k$ with a total of points of $E(n_k) = n \times l_k/h$ of the time slot and $l_k = |A_k|$. 
The bridge arc at the end of each time slot $A_k$ is considered to be equivalent to the arc closing the optimal Hamiltonian path of points of $A_k$ in order to be a circuit. 
Since the BHH theorem does not require knowing the specific positions of points to visit in addition to the assumption (iv), we could make use of this assumption.
Thus, we have
$$ L_{n}^{tsp-mts} = \sum_{k=1}^{m} L_{E(n_k)}^{tsp} 
= \beta^{tsp} \;\sqrt{\frac{n}{h}|\mathcal{R}|} \sum_{k=1}^{m} \sqrt{l_k}
= L_{n}^{tsp} \sum_{k=1}^{m} \sqrt{\frac{l_k}{h}}.$$
\end{proof}

In case of only one-time slot equal to the time horizon $[0,h]$, the TSP-$m$TS approximation corresponds to the TSP formula $L_{n}^{tsp-1ts} = L_{n}^{tsp}$.
In case of identical time slots, {\em i.e.} $l_k= h/m$, the satisfiability condition (\ref{eq:pareto2}) and the approximation (\ref{eq:tsp-mts}) correspond respectively to (\ref{eq:pareto}) and (\ref{eq:tsp-mits}). \\

\begin{prop}[\textbf{Bound of TSP-$m$TS}]
\label{prop5}
For the TSP-mTS random generation model with the time slots $(A_k)_{1 \leq k\leq m}$,
the asymptotic length $L_{n}^{tsp-mts}$ of the the tour linking $X_0, ...X_{n-1}$ under the  {\em Euclidean} metric in the
case of feasibility satisfies
\begin{equation}\label{bd:tsp-mts}
L_{n}^{tsp} \leq L_{n}^{tsp-mts} \leq L_{n}^{tsp-mits}.
\end{equation}
\end{prop}

\begin{proof}
Given the lengths of the time slots $(l_k)_k$, the lower bound of (\ref{bd:tsp-mts}) is obtained by using the following inequality 
\begin{equation}\label{bd:tspdtw-1}
\sqrt{\sum_{k=1}^{m}l_k} \leq \sum_{k=1}^{m} \sqrt{l_k} 
\end{equation}
which can be proved by recursion. 
Using (\ref{eq:tsp-mts}) and $h=\sum_{k=1}^{m} l_k$, we have
$$
(\ref{bd:tspdtw-1}) 
\Longrightarrow \beta^{tsp} \;\sqrt{\frac{n}{h}|\mathcal{R}|} \sqrt{\sum_{k=1}^{m} l_k} \leq L_{n-1}^{tsp-mts} 
\Longrightarrow \beta^{tsp} \;\sqrt{n|\mathcal{R}|}= L_{n}^{tsp} \leq L_{n}^{tsp-mts}.
$$

The upper bound of (\ref{bd:tsp-mts}) is obtained by using the generalized mean inequality, specifically the inequality between the arithmetic and quadratic mean, which gives 
$$
\sum_{k=1}^{m} \sqrt{l_k} \leq \sqrt{m \sum_{k=1}^{m} l_k}=\sqrt{m\times h}.
$$
From (\ref{eq:tsp-mits}), the upper bound is then equal to
$$ \beta^{tsp} \;\sqrt{\frac{n}{h}|\mathcal{R}|} \;\sqrt{m\times h}= L_{n}^{tsp-mits}.$$
\end{proof}

\subsection{The TSP-TW}
In this subsection, we provide a general bound of the TSP-TW asymptotic approximation, based on an induced TSP-TS. We first define the time windows induced time slots that are used in the induced TSP-TS problem.\\

\begin{definition}[\textbf{TW induced TS}]
\label{def1}
Given the time windows $[b_i, f_i]$ of the set of points $\cal P$ and a time horizon $h$, we define the induced time slots as follows
\begin{enumerate}
    \item Sort all begin and end TW bounds of ${\cal P}$  by ascending order, from the set ${\cal S}_0 = \{b_1, f_1, ..., b_{n-1}, f_{n-1}\}$ to obtain the set ${\cal S}_1=\{c_0, c_1, ..., c_{2n-1}\}$, where $c_0= 0$ and $c_{2n-1}=h$.
    \item Define the time slots according to the ${\cal S}_1$ order, such that $A_{k}=[c_{k-1}, c_{k}],$ $k\in \{1,..,2n-1\}$. The $(2n-1)$ time slots $(A_k)$ constitute a partition of the time horizon. 
\end{enumerate}
\end{definition}
The total number of time slots is $2n-1$. 
This number can be decreased by one for each $c_{k-1}=c_{k},$ for all $1 \leq k\leq 2n-1$, which represent empty time slots.
It can be further decreased by removing time slots with no client assigned to them. 
In fact, each client with a time window $[b_i, f_i]$ can be served in one of the time slots $A_k$, where $j+1\leq k\leq j+l$, $c_j=b_i$, and $ c_{j+l} = f_i$, with $l \geq 1$, are the correspondent points of $b_i$ and $f_i$ in the ordered set ${\cal S}_1$.
Let us denote the final number of time slots as $m^*= 2n-1-m_1-m_2$, where $m_1$ is the number of duplicated $c_k$, and $m_2$ is the number of time slots with no client.\\

\begin{prop}[\textbf{Bound of TSP-TW}]
\label{prop6}
For the TSP-TW random generation model, we have 
\begin{equation}\label{bd:tsptw}
    L_{n}^{tsp} \leq L_{n}^{tsp-tw} \leq L_{n}^{tsp-m^* its},
\end{equation}
where $m^*= 2n-1-m_1-m_2$, with $m_1$ is the number of duplicated time windows bounds of ${\cal P}$, and $m_2$ is the number of time slots with no client affected to them.
\end{prop}
\begin{proof}
The lower bound can be easily derived.
Given a random instance $p$ of the TSP-TW,
let $r$ be the optimal route of the corresponding TSP, {\em i.e.} same problem configuration of the TSP-TW without the time windows, with a duration $d_r$. 
For all feasible tours $r'$ of $p$, 
we have by definition $d_r \leq d_{r'}$. 
Then, asymptotically $L_{n}^{tsp} \leq L_{n}^{tsp-tw}$.

For the upper bound, we consider the induced TSP-$m^*$TS problem $p'$ of $p$, which corresponds to the problem with the same settings of $p$ apart for time windows considerations. 
They are taken to be the time slots of Definition~\ref{def1}.
The approximation of the TSP-$m^*$TS does not require knowing the exact assignment of points to the time slots (see Proposition ~\ref{prop4}), in the similar way that the BHH formula does not require knowing the exact spatial locations of ${\cal P}_{all}$. 
Suppose $r^*$ is the optimal solution of the TSP-TW instance $p$.
This solution can be reached for one specific assignment of ${\cal P}$ to the induced time slots.
Actually, the problem $p'$ is more constrained than $p$, as each point is assigned to a time slot smaller or equal to its original time window. 
Thus, we have 
\begin{equation}\label{bd:tsptw1}
    L_{n}^{tsp-tw} \leq L_{n}^{tsp-m^*ts}.
\end{equation}

Across all time slots configurations, the worst case is attained for identical TS, according to Proposition~\ref{prop5}:
 $$L_{n}^{tsp-tw} \leq L_{n}^{tsp-m^*its}.$$

Concerning the problem satisfiability, if $p$ is not feasible, $p'$ will be as well as it is more constrained:
$ L_{n}^{tsp-tw} =+\infty  \Longrightarrow L_{n-1}^{tsp-m^*its}=+\infty.$
\end{proof}

The inequality (\ref{bd:tsptw}) implies  
$$ 1 \leq \frac{L_{n}^{tsp-tw}}{L_{n}^{tsp}} \leq  \frac{L_{n}^{tsp-m^*its}}{L_{n}^{tsp}} \leq  \frac{L_{n}^{tsp-(2n-1)its}}{L_{n}^{tsp}} = \mathcal{O}(\sqrt{n}),$$ 
meaning that the upper bound ratio over the TSP approximation has a slow growth rate. 

However, the upper bounds (\ref{bd:tsptw}) and (\ref{bd:tsptw1}) are generally not tight, as the number of points assigned to each time slot can be quite small, because asymptotic approximations are unsuitable when the number of points is small.
At the opposite, when the points ${\cal P}$ have large time windows and are densely located in time, the upper bounds can be of a better quality.\newline 

\textbf{Remark.} to provide an approximation of the TSP-TW, one way could be to relax some bounds of the time windows such that the TW-induced TS can be large enough to encompass several points, more than the strict initial configuration.  
By doing so, this approximation is not an upper bound anymore, and is rather an estimation of the TSP-TW asymptotic approximation.
Which time windows bounds of ${\cal P}$ to consider for relaxation is still an open question.
The estimation of the asymptotic approximation can be given by
$
L_{n}^{tsp-tw} \approx L_{n}^{tsp-mts} = \beta^{tsp} \; \sqrt{\frac{n\;|\mathcal{R}|}{h}} \sum_{k=1}^{m}\sqrt{l_k},
$
where $(l_k)_{1 \leq k\leq m}$ are the lengths of the relaxed induced time slots of the TSP-TW problem.

\section{The worst-case demand}\label{sec:wc}
In this section, we relax the uniform distribution hypothesis of the instances for more realistic consideration.
We first explicit the motivations and the choice of the distribution,
then provide the expression of these distributions for both time and space.
Finally, the satisfiability condition and the asymptotic approximation of the TSP-$m$ITS are given accounting for those distributions. 

\subsection{Motivations and the choice of the distribution}
The approximations of the previous section rely on  uniform distributions, which assign equal probabilities to customers' requests whether in terms of space or time (distribution on the time slots).
This is rather a general assumption.
The distribution of customers' demand is known to vary considerably during the day, with a low demand during the noon period for instance.
Variations are also spatial as city centers and high-density metropolitan areas attract more demands.
If we happen to learn some features of the demand distribution, which is generally the case of companies and services using historical data, being able to include this additional information in the choice of the temporal and the spatial demand distributions is beneficial in order to improve the asymptotic approximations of the routing problems.

The actual demand distribution (for which the routing problem is solved) is  not completely known in advance, although some partial statistical features could be discovered.
To face this underlying uncertainty, predictive models can be used to construct the global demand distribution, as done in \cite{iglesias2018data}. 
Another approach follows the distributionally robust optimization, in which optimization is performed against the worst-case realization of the unknown distribution \cite{rahimian2019distributionally}.
This approach falls under the paradigm of robust optimization \cite{ben2009robust}, which becomes popular for optimizations under uncertainty as it constructs computationally tractable robust counterparts of problems.
Several statistical methods exist to estimate the worst-case distribution: the maximum likelihood estimator (MLE) \cite{myung2003tutorial}, the principle of maximum entropy (ME) \cite{conrad2004probability}, and the minimum Hellinger distance estimator \cite{beran1977minimum}.
Carlson and Behzoodi \cite{carlsson2017worst} have generated a worst-case spatial demand distribution for capacitated VRP by maximizing the upper bound of the VRP tour length given by Haimovich and Rinnoy Kan \cite{haimovich1985bounds}. 
We opted for the maximum entropy distribution due to its intuitive nature, and its higher probabilistic properties compared to the MLE as smoothing techniques are not needed\footnote{The MLE and the ME approaches are equivalent for the estimation of natural discrete distributions given a number of moments \cite{golan1998maximum}.}. 

The entropy introduced by Shannon \cite{shannon2001mathematical} is a measure of the degree of disorder or uncertainty in a random variable.
A low entropy indicates that the system representing the random variable is organized, can be easily predicted, and needs fewer bits to be represented. At the opposite, a high entropy system tends to be highly dispersed, difficult to predict, and needs considerable amounts of information in order to be stocked.
In the event of unknown demand, we could argue that the worst case corresponds to a notoriously difficult to predict scenario in terms of customers' appearance in both time and space. 
The element of surprise during the transportation operation has to be at its maximum level.
Subsequently, the requests of customers would tend to follow and be induced by a geographical and temporal distribution with high entropy.
In this respect, maximum entropy distributions on the studied area $\mathcal{R}$ and for time windows consideration are a suitable choice  for worst case demand distributions.

The estimation method of ME distribution is based on maximizing of the entropy measure under a given number of first moments.
If no moment constraint is imposed, the distribution with the highest entropy corresponds to the uniform distribution. We constrain our study to the first two moments, and to the study of TSP-ITS case. 

\subsection{Worst-case spatial demand}
The spatial demand distribution $f(.)$ is a continuous two dimensional distribution for which the mean $\mu$ and the covariance matrix $\Sigma \succ 0$ are given. 
$f(.)$ is defined on the square area $\mathcal{R}=[0,a] \times [0,a]$.
The optimization problem aimed to find $f(.)$ is as follow
\begin{align} 
\max_{f(.)} \quad & H(f)= - \int_{\mathcal{R}}f(x,y)\;\log f(x,y)\; dx\; dy, \label{opt1-1}\\ 
\textrm{s.t.} \quad & \mathbf{\mu} = \int_{\mathcal{R}} \icol{x\\y} \;f(x,y)\; dx\;dy \label{opt1-2}\\
& \Sigma + \mathbf{\mu}\mathbf{\mu}^T = \int_{\mathcal{R}} \icol{x\\y}\icol{x\\y}^T \;f(x,y)\; dx\;dy \label{opt1-3}\\
&\int_{\mathcal{R}}f(x,y)\; dx\;dy = 1  \label{opt1-4} \\
  &f(x,y)\geq 0\quad \quad \forall (x, y) \in \mathcal{R},   \label{opt1-5}
\end{align}
where $(\ref{opt1-1})$ is the entropy function, $(\ref{opt1-2})$ is the constraint on the first moment $\mu$ of $f(.)$, and $(\ref{opt1-3})$ is the constraint on the second moment of $f(.)$, the covariance matrix $\Sigma$.
The probability distribution constraints are given by $(\ref{opt1-4})$ and $(\ref{opt1-5})$. 
Applying the Lagrange multipliers method to this problem leads to the following exponential density function, 
\begin{equation}\label{dens1}
    f(x,y)= \exp \{\nu-1+\lambda^T\icol{x\\y}+ \icol{x\\y}^T Q_f \icol{x\\y}\},
\end{equation}
where $\nu\in\mathbb{R}$, $\lambda\in\mathbb{R}^2$, and $Q_f\succeq 0$ are Lagrange multipliers, 
which are associated respectively to the constraints (\ref{opt1-4}), (\ref{opt1-2}) and  (\ref{opt1-3}).
These multipliers can be determined by solving the equations 
(\ref{eqLagrange1})-(\ref{eqLagrange3}) corresponding to the constraints (\ref{opt1-2})-(\ref{opt1-4}) in which the function $f(.)$ of (\ref{dens1}) is replaced:
\begin{align} 
&\mathbf{\mu} = \int_{\mathcal{R}} \icol{x\\y} \;e^{\nu-1+\lambda^T\icol{x\\y}+ \icol{x\\y}^T Q_f \icol{x\\y}}\; dx\;dy \label{eqLagrange1}\\
& \Sigma + \mathbf{\mu}\mathbf{\mu}^T = \int_{\mathcal{R}} \icol{x\\y}\icol{x\\y}^T \;e^{\nu-1+\lambda^T\icol{x\\y}+ \icol{x\\y}^T Q_f \icol{x\\y}}\; dx\;dy \label{eqLagrange2}\\
&\int_{\mathcal{R}}e^{\lambda^T\icol{x\\y}+ \icol{x\\y}^T Q_f \icol{x\\y}}\; dx\;dy = e^{1-\nu}   \label{eqLagrange3}
\end{align}

The maximum entropy distribution (\ref{dens1}) exists and is unique for one dimension, see \cite{dowson1973maximum}. 
In case the support of the distribution $\mathcal{R}$ is the real domain $\mathbb{R}^2$, the ME distribution corresponds to the bi-variate normal distribution \cite{conrad2004probability}.  

\subsection{Worst-case temporal demand}
The ME distribution for the temporal demand follows the same idea of the spatial counterpart, expect that the temporal density function $g(.)$ is discrete and univariate.
We split the time horizon to equal size time slots, which represent the periods of the day. 
The worst-case computation here is done on the position of the time slot.
This is why we omitted the case of TS with various lengths.
The support of the distribution $g(.)$ is $\mathcal{S}=\{0,1, .. m\}$, which incorporates the $m$ time slots and the case of not choosing any TS, {\em i.e.} the case 0. 
We let $g_i=g(l_i),\;\;\forall i = 0, ..,m$, 
The problem to optimize is as follows

\begin{align}
\max_{g} \quad & H(g)= - \sum_{i=0}^{m} g_i\;\log g_i, \label{opt2-1}\\ 
\textrm{s.t.} \quad & 
\mu = \sum_{i=0}^{m} i \times g_i \label{opt2-2}\\
& \sigma^2 + \mu^2 = \sum_{i=0}^{m} i^2 \times g_i \label{opt2-3}\\
&\sum_{i=1}^{m} g_i = 1  \label{opt2-4} \\
&g_i\geq 0\quad \quad \forall i = 0, ..,m. \label{opt2-5}
\end{align}

This problem without the constraint (\ref{opt2-3}) of the variance $\sigma^2$ has a unique solution corresponding to the binomial distribution \cite{harremoes2001binomial}.
For simplicity of derivation, compared to solving for all constraints, we constrain the study  to the first moment $\mu$.
This value will provide an indication of the median TS. 
For instance, if the distribution $g(.)$ is uniform, then $\mu = (m+1)/2$.

\begin{table}[hptb]
    \centering
\noindent \begin{tabular}{c|c|c|}
f(.) $\backslash$ g(.)  & moment constraints: None  & $\mu_{g}=m\;p_{b}$ \\\hline
&&\\
None &  uniform distr. f(.)& uniform distr. f(.)\\
& uniform distr. g(.)& binomial distr. g(.) \\ \hline
&&\\
$\mu_{f}$ and $\Sigma_{f}$ & exponential distr. f(.) &  exponential distr. f(.) \\  
&uniform distr. g(.)&  binomial distr. g(.)\\\hline
\end{tabular} 
\caption{Four worst-case distribution, where the first moment of $g(.)$ is $\mu_{g}=m\;p_{b}$, with $p_{b}$ is the probability of successes of the binomial distribution, and $\mu_{f}$ and $\Sigma_{f}$ are respectively the first and second moment of $f(.)$.}
\label{tab2}
\end{table}

\subsection{Worst-case asymptotic approximations}
Depending on the moment values for the spatial $f(.)$ and the temporal $g(.)$ demand distributions, four combinations of the worst-case distribution can be given, as shown in Table~\ref{tab2}. Subsequently, four approximations can be stated for the tour length of the TSP-$m$ITS case.
The equivalent version of the satisfiability condition is given as follows. \newline


\noindent \textbf{Satisfiability TSP-$m$ITS under ME}
For the TSP-$m$ITS random generation model, wherein the demand distributions f(.) and g(.) follow the ME principle, the satisfiability condition of feasible tours linking the realisations of $X_0, ...X_{n-1}$ under the {\em Euclidean} metric satisfies on average,
\begin{center}
 \begin{tabular}{c|c|c|}
f(.) $\backslash$ g(.)  & None  & $\mu_{g}$ \\\hline
&&\\
None &  $n\times m  = $ & $n\times f_1(\mu_g,m)  = $ \\
&$h^2 /\big((\beta^{tsp})^2\;|\mathcal{R}|\big)$ &  $h^2 /\big((\beta^{tsp})^2\;|\mathcal{R}|\big)$\\ \hline
&&\\
$\mu_{f}$ and $\Sigma_{f}$ & $n\times m = $ &  $n\times f_1(\mu_g,m) = $ \\  
& $h^2 /\big(\beta^{tsp}\;F(\mu_{f}, \Sigma_{f})\big)^2$ 
& $h^2 /\big(\beta^{tsp}\;F(\mu_{f}, \Sigma_{f})\big)^2$  \\ \hline
\end{tabular}
\end{center}
\noindent such that, $\mu_g\in [|1,k|]$,
and $$F(\mu_{f}, \Sigma_{f})=\int_{\mathcal{R}} e^{(\nu-1+\lambda^T\icol{x\\y}+ \icol{x\\y}^T Q_f \icol{x\\y})/2} dx dy,$$ where $\nu\in\mathbb{R}$, $\lambda\in\mathbb{R}^2$ and $Q_f\succeq 0$ satisfy (\ref{eqLagrange1}), (\ref{eqLagrange2}) and (\ref{eqLagrange3}), and  
$$ f_1(\mu_g,m) = m^2\times{m \choose \mu_g}\big(\frac{\mu_g}{m}\big)^{\mu_g} \big(1-\frac{\mu_g}{m}\big)^{m-\mu_g}.$$

A similar analysis to the satisfiability condition (\ref{eq:pareto2}) of the uniform case is used here, tacking into account the new expression of the building block of the TSP asymptotic approximation modified under ME. 
When the constraint $\mu_{f}$, $\Sigma_{f}$, and $\mu_{g}$ are imposed, the BHH formula for the $k$-th time slot becomes  
$$
\frac{L_{E(n_k)}^{tsp}}{\sqrt{E(n_k)}}  \xrightarrow[n\to \infty]{} \beta^{tsp} \;F(\mu_{f}, \Sigma_{f}),
$$
where $E(n_k)$ is the expected number of points assigned to the $k$-th time slot.
Let $Prob(k;m;p_b)$ be the probability mass function of the binomial distribution assigning a point to a time slot $k$, such that $p_b=\mu_g/m$ is the  probability of success.
We have 
\begin{align*}
    E(n_k) &= n \times Prob(k;m;p_b) 
    =n \times {m \choose k} \big(\frac{\mu_g}{m}\big)^k \big(1-\frac{\mu_g}{m}\big)^{m-k}.
\end{align*}
$E(n_k)$ is maximized for $k=p_b\;m=\mu_g$.
The equality to obtain the satisfiability condition is the following
$$ L_{E(n_k)}^{tsp}= \beta^{tsp} \;F(\mu_{f}, \Sigma_{f})\sqrt{n\;Prob(\mu_g;m;\frac{\mu_g}{m})}= h/m.$$
In case $\mu_{g}$ is not imposed, then $E(n_k)=n/m$. We can solve for this value.
In case $\mu_{f}$, $\Sigma_{f}$ are not imposed, then $F(\mu_{f}, \Sigma_{f}) = \sqrt{|\mathcal{R}|}.$\\

\begin{prop}[\textbf{Approximation TSP-$m$ITS under ME}]
\label{prob-approx-wc}
For the TSP-$m$ITS random generation model, wherein the distributions of f(.) and g(.) follow the ME principle,
the asymptotic optimal length $L_{n-1}^{tsp-mits}$ of the tour linking $X_0, ...X_{n-1}$ under the  {\em Euclidean} metric if feasible is equal to
\begin{center}
\noindent \begin{tabular}{c|c|c|}
f(.) $\backslash$ g(.)   & None  & $\mu_{g}$ \\\hline
&&\\
None & $\beta^{tsp}  \sqrt{n\;m\;|\mathcal{R}|}$ & $\beta^{tsp} \; \sqrt{n\;f_2(\mu_g,m)\;|\mathcal{R}|}$  \\ \hline
&&\\ 
$\mu_{f}$ and $\Sigma_{f}$ & $\beta^{tsp}  \sqrt{n\;m}\; F(\mu_{f}, \Sigma_{f})$ & $\beta^{tsp} \; \sqrt{n\;f_2(\mu_g,m)}\; F(\mu_{f}, \Sigma_{f})$ \\  \hline
\end{tabular}
\end{center}

\noindent such that, $\mu_g\in [|1,k|]$, $n/Q \in \mathbb{N}$, $l\times Q \geq n$, and
$$F(\mu_{f}, \Sigma_{f})=\int_{\mathcal{R}} e^{(\nu-1+\lambda^T\icol{x\\y}+ \icol{x\\y}^T Q_f \icol{x\\y})/2} dx dy,$$ 
where $\nu\in\mathbb{R}$, $\lambda\in\mathbb{R}^2$ and $Q_f\succeq 0$ satisfy (\ref{eqLagrange1}), (\ref{eqLagrange2}) and (\ref{eqLagrange3}), and 
$$  f_2(\mu_g,m) = \sum_{k=1}^{m}\sqrt{{m \choose k} \big(\frac{\mu_g}{m}\big)^{k} \big(1-\frac{\mu_g}{m}\big)^{m-k}}. 
$$
\end{prop}

\begin{proof} Similar proof to that of proposition~\ref{prop4}. When the constraint $\mu_{f}$, $\Sigma_{f}$ and $\mu_{g}$ are imposed, the approximation is equal to the following summation 

\begin{align*}
    L_{n}^{tsp-mits} &= \sum_{k=1}^{m} L_{E(n_k)-1}^{tsp} = \beta^{tsp}\;F(\mu_{f}, \Sigma_{f}) \sum_{k=1}^{m}  \sqrt{n\;Prob(k;m;\frac{\mu_g}{m})}\\
&= \beta^{tsp}\;F(\mu_{f}, \Sigma_{f})\sqrt{n} \sum_{k=1}^{m} \sqrt{{m \choose k} \big(\frac{\mu_g}{m}\big)^k \big(1-\frac{\mu_g}{m}\big)^{m-k}},
\end{align*} 
where $E(n_k)$ is the expected number of points assigned to the $k$-th time slot, and $Prob(k;m;\mu_g/m)$ is the probability mass function of the binomial distribution assigning a point to visit to the $k$-th time slot, where $p_b=\mu_g/m$ is the  probability of success.
\end{proof}
Notice that through arithmetic mean - quadratic mean (AM-QM) inequality, we have $f_2(\mu_g,m) \leq \sqrt{m}$, meaning that the TSP-$m$ITS approximation when the mean $\mu_g$ is imposed is always lesser or equal to that for uniform temporal distribution.

\section{Solving approach}
Solving approaches of the TSP-TW are mostly designed to take advantage of the tightness of the time windows, {\em e.g.} \cite{baldacci2012new}. 
They do not perform efficiently in case the time window constraints are quite large when for instance equaling half or a third of the time horizon. 
The most efficient relaxations of the TSP-TW, namely $t$-Tour and $ng$-Tour relaxations, take advantage of the tightness of the time windows in order to calculate higher lower bounds that are used in the branch-and-price approaches. 
Since the time windows of the TSP-TS have a specific structure of non-overlapping, we can take advantage of this property in order to propose an adapted approach to solve the problem to optimality.
In our approach, we first construct an acyclic directed graph, which uses a TSP solver to obtain the lengths of Hamiltonian paths inside each time slot. 
In the second step, we run a shortest path algorithm on this graph to generate the optimal tour for the instance. In this way, we could exploit the non-overlapping property of time windows.

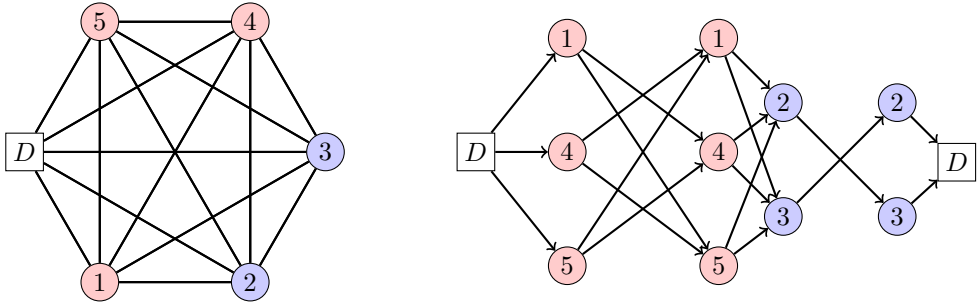
\begin{figure}[h!]
  \centering
  \begin{tikzpicture}
    \node[draw,minimum size=0.5cm,inner sep=0pt] (D) at ({360/6 * (4- 1)}:2cm) {$D$};
    \node[circle,fill=red!20,draw,minimum size=0.5cm,inner sep=0pt] (1) at ({360/6 * (5- 1)}:2cm) {$1$};
    \node[circle,fill=blue!20,draw,minimum size=0.5cm,inner sep=0pt] (2) at ({360/6 * (6- 1)}:2cm) {$2$};
    \node[circle,fill=red!20,draw,minimum size=0.5cm,inner sep=0pt] (5) at ({360/6 * (3- 1)}:2cm) {$5$};
    \node[circle,fill=red!20,draw,minimum size=0.5cm,inner sep=0pt] (4) at ({360/6 * (2- 1)}:2cm) {$4$};
    \node[circle,fill=blue!20,draw,minimum size=0.5cm,inner sep=0pt] (3) at ({360/6 * (1- 1)}:2cm) {$3$};

    \path[draw,thick]
    (D) edge node {} (1)
    (D) edge node {} (2)
    (D) edge node {} (3)
    (D) edge node {} (4)
    (D) edge node {} (5)
    
    (1) edge node {} (D)
    (1) edge node {} (2)
    (1) edge node {} (3)
    (1) edge node {} (4)
    (1) edge node {} (5)
    
    (2) edge node {} (1)
    (2) edge node {} (D)
    (2) edge node {} (3)
    (2) edge node {} (4)
    (2) edge node {} (5)
    
    (3) edge node {} (1)
    (3) edge node {} (2)
    (3) edge node {} (D)
    (3) edge node {} (4)
    (3) edge node {} (5)
    
    (4) edge node {} (1)
    (4) edge node {} (2)
    (4) edge node {} (3)
    (4) edge node {} (D)
    (4) edge node {} (5)
    
    (5) edge node {} (1)
    (5) edge node {} (2)
    (5) edge node {} (3)
    (5) edge node {} (4)
    (5) edge node {} (D);

    \begin{scope}[xshift=4cm]
    \node[draw,minimum size=0.5cm,inner sep=0pt] (D_1) {$D$};

    \node[circle,fill=red!20,draw,minimum size=0.5cm,inner sep=0pt] (4_1) [right= 0.7cm of D_1]  {$4$};
    \node[circle,fill=red!20,draw,minimum size=0.5cm,inner sep=0pt] (1_1) [above = 1cm of 4_1] {$1$};
    \node[circle,fill=red!20,draw,minimum size=0.5cm,inner sep=0pt] (5_1) [below =1cm of 4_1]  {$5$};
    \node[circle,fill=red!20,draw,minimum size=0.5cm,inner sep=0pt] (1_2) [right=1.5cm of 1_1]  {$1$};
    \node[circle,fill=red!20,draw,minimum size=0.5cm,inner sep=0pt] (4_2) [right=1.5cm of 4_1] {$4$};
    \node[circle,fill=red!20,draw,minimum size=0.5cm,inner sep=0pt] (5_2) [right=1.5cm of 5_1] {$5$};
    \node[circle,fill=blue!20,draw,minimum size=0.5cm,inner sep=0pt] (2_1) [below right=0.7cm of 1_2] {$2$};
    \node[circle,fill=blue!20,draw,minimum size=0.5cm,inner sep=0pt] (3_1) [below=1cm of 2_1] {$3$};
    \node[circle,fill=blue!20,draw,minimum size=0.5cm,inner sep=0pt] (2_2) [right=1cm of 2_1] {$2$};
    \node[circle,fill=blue!20,draw,minimum size=0.5cm,inner sep=0pt] (3_2) [right=1cm of 3_1] {$3$};
    \node[draw,minimum size=0.5cm,inner sep=0pt] (D_2)  [below right=0.5cm of 2_2]  {$D$};

    \path[->, draw,thick]
    (D_1) edge node {} (1_1)
    (D_1) edge node {} (4_1)
    (D_1) edge node {} (5_1)
    (1_1) edge node {} (4_2)
    (1_1) edge node {} (5_2)
    (4_1) edge node {} (1_2)
    (4_1) edge node {} (5_2)
    (5_1) edge node {} (1_2)
    (5_1) edge node {} (4_2)
    (1_2) edge node {} (2_1)
    (1_2) edge node {} (3_1)
    (4_2) edge node {} (2_1)
    (4_2) edge node {} (3_1)
    (5_2) edge node {} (2_1)
    (5_2) edge node {} (3_1)
    (2_1) edge node {} (3_2)
    (3_1) edge node {} (2_2)
    (2_2) edge node {} (D_2)
    (3_2) edge node {} (D_2);
    \end{scope}
    
    \end{tikzpicture}
    \caption{An example of the DAG of a TSP-TS instance of $m=2$ time slots. The red (resp. blue) color represents points of the first (resp. second) time slot. $D$ is the depot.}
    \label{fig:expl_graph}
\end{figure}

\begin{algorithm}[h!]
\caption{The DAG construction algorithm.}\label{algo1}
\KwData{a TSP-TS instance ($n$=number of points to visit,   $m$=number of time slots, $c_{ij}$=travel cost from point $i$ to point $j$, $ts_i$=the corresponding time slots of point $i$).}
\KwResult{A directed acyclic graph $G$.}
$V\longleftarrow  \emptyset$\;
$E \longleftarrow  \emptyset$\;
$n_G \longleftarrow 2 \times n$\;
Add vertices $v_1(=0)$ and $v_{n_G}(=0)$ to $V$ (depot=0)\;
$l \longleftarrow $ 2\;
\For{time slot $k$ in $[1, m]$}{
\For{points $i, j$ s.t. $TS(i)=TS(j)=k$ and $i \neq j$}{
Add vertices $v_{l}(=i)$ and $v_{l+1}(=j)$ to $V$\; 
Add arcs $(v_{l}, v_{l+1})$ to $E$ with the cost:\:
Dist[$v_{l}$ ; $v_{l+1}$] $\longleftarrow$ cost of Hamiltonian path  from $v_{l}$ to $v_{l+1}$\;
}
}
Add arcs $(v_1, v_i)$ to $E$, for all points $v_i$ s.t. $TS(v_i) =1$ with the costs Dist[$v_1$ ; $v_i$] $=c_{0, j}$ where $j=Idx(v_i)$\;
\For{time slot $k$ in $[1, m]$}{
Add arcs $(v_i, v_j)$ to $E$, for all points $v_i$ and $v_j$ s.t. $v_i$ has no outgoing arc, $v_j$ has no incoming arc,  $TS(v_i) =k$ and $TS(v_j) =k+1$ with the costs Dist[$v_i$ ; $v_j$] $=c_{j1, j2}$ where $j1=Idx(v_i)$ and $j2=Idx(v_j)$\;
}
Add arcs $e_{v_i, v_{n_G}}$ to $E$, for all points $v_i$ s.t. $TS(v_i) =m$ with the costs Dist[$v_i$ ; $v_{n_G}$] $=c_{j, 0}$ where $j=Idx(v_i)$\;
return $G=(V, E)$\;
\end{algorithm}

The construction of the graph of a TSP-TS instance is given by Algorithm~\ref{algo1}.
It outputs a directed acyclic graph (DAG), with one unique source and one unique sink vertex. Both correspond to the depot of the input instance. 
The travel cost from $i$ to $j$ of an arc in the graph, termed Dist[i,j],  corresponds to the original instance cost $c_{ij}$ if the points have different time slots, {\em i.e.} $TS(j)=TS(i)+1$. 
In case the points $i$ and $j$ belong to the same time slot, {\em i.e.} $TS(j)=TS(i)$, Dist[i,j] corresponds to the shortest Hamiltonian path length going from $i$ to $j$ and passing through all the remaining points of the time slot.
Arcs from $i$ to $j$ such as $TS(i) \neq TS(j)$ and  $TS(j) \neq TS(i)+1$ are not considered in our graph as they correspond to unfeasible paths. 
The function $Idx(v)$ tracks the index of a vertex $v$ in the graph  in the original instance.
An example DAG graph is given in Figure~\ref{fig:expl_graph}.

\begin{algorithm}[h!]
\caption{The TSP-TS shortest path resource constraint algorithm.}\label{algo}
\KwData{G(V,E): the DAG of the input instance given by Algorithm~\ref{algo1}.}
\KwResult{The shortest path length from source(G) to sink(G).}
 BKS = GreedyPath(G)\;
 N = CreateSubPath(PathLength=0, ArriveTime=0, Binf=$-\infty$, LastVertex=source(G))\;   
 Q = \{N\}\;
 \While{$Q \neq \emptyset$}{
  N = Pop(Q)\;
  i $\longleftarrow$ LastVertex(N)\;
  \For{all edges $(i,j) \in E$ }{
  M = CreateSubPath(PathLength=$\infty$, ArriveTime=$\infty$, Binf=Binf(N), LastVertex=j)\;
  ArriveTime(M) = max( ArriveTime(N) + Dist[i ; j] ; StartTime(TS(j)) )\;
  \If{ArriveTime(M) $\leq$ FinishTime(TS(j))}{
        PathLength(M) = PathLength(N) + Dist[i ; j]\;
		Binf(M) = PathLength(M) + $\sum_{p= NV(j)}^{ NV_{max}-1} \text{Dist}_{min}$[p ; p+1]\;
		\If{ Binf(M) $\leq$  BKS}{
				\If{ M is not dominated in Q }{
				Insert(M, Q)\;
				\If{NV(j) = $NV_{max}$ and PathLength(M) $<$ BKS}{
						BKS = PathLength(M)
					}
				}
			}
    }
  }
 }
 return BKS\;
\end{algorithm}

The shortest path computation in the DAG $G$ of the instance to output the optimal tour length is described in Algorithm~\ref{algo}.
It finds the shortest path from the source of the graph, $source(G)$, to its sink, $sink(G)$, which constitutes a Hamiltonian circuit starting and ending at the depot. Each sub-path $N$ has the 4 attributes: the last vertex visited, $LastVertex$, the optimal path length from $source(G)$ to this vertex, $PathLength$, the arrival time to this vertex, $ArriveTime$, and a lower bound value $Binf$. 
The main data structure of Algorithm~\ref{algo} is a list of list of sub-paths, indexed by $LastVertex$ and sorted by increasing $PathLength$. It is denoted  $Q$.
The algorithm includes also a mechanism of dominance and cutting by lower bounds, which allows us to not explore all dominated sub-paths and their extensions.
A sub-path $M$ with $LastVertex(M)=j$ is said to be dominated in case it exists a sub-path $L$ in $Q$ such that $LastVertex(L) = j$, $PathLength(L) \leq PathLength(M)$ and $ArriveTime(L) \leq ArriveTime(M)$.
The lower bound of a sub-path $M$, $Binf(M)$, is calculated by completing $PathLength(M)$ with the lowest cost value from $LastVertex$ to $sink(G)$, without connectivity considerations.
To do so, a minimal distance between each successive graph levels has to be previously computed.
The minimal distance between a level $p$ and its next level $p+1$ is given by the function $\text{Dist}_{min}\text{[p ; p+1]} = \min_{NV(k) = p,\; NV(m) = p+1} \text{Dist[k ; m]}$.
The following functions are also used in Algorithm~\ref{algo}:  
\begin{itemize}
    \item[-] $NV(v)$: the level of the vertex $v$ in the DAG $G$. The maximal level of $G$ is denoted $NV_{max}= \max_{v \in V} NV(v)$, which corresponds to the level of $sink(G)$.
    \item[-] TS(v): the time slot of the vertex $v$ in $G$.
    \item[-] StartTime(ts): the start time of a time slot ts. 
    \item[-] EndTime(ts): the end time of a time slot ts 
    \item[-] GreedyPath(G): the greedy growing path from $source(G)$ to $sink(G)$. 
\end{itemize}

\section{Computational results}\label{sec:expe}
In this section, we report experimental results that examine the quality of the proposed asymptotic approximation on a number of benchmarks of the TSP-TW altered in this occasion for time slots considerations.

\subsection{Benchmark generation}
To generate instances of the TSP-TS, we make use of the benchmarks of the TSP-TW listed in Table~\ref{tab:ben}.
The table encompasses benchmark datasets that are often used in the literature for evaluating solving algorithms of the TSP-TW\footnote{These datasets can be downloaded from the websites:  \url{https://myweb.uiowa.edu/bthoa/tsptwbenchmarkdatasets.htm}, and \url{http://lopez-ibanez.eu/tsptw-instances} 
[Accessed 2022-05-29].}.
The column $n-1$ designates the number of customers to visit, and $w$ is the maximum width used in the generation of the lengths of the time windows. 
Our goal is to take the physical distribution of the points to visit and the depot from these instances, and change their time windows configuration to time slots.

\begin{sidewaystable}
    \centering
    \begin{tabular}{|c|c|c|c|c|c|c|c|c|c|c|c|}
        
        \multicolumn{1}{c}{1} & \multicolumn{1}{c}{2} & \multicolumn{1}{c}{3} & \multicolumn{1}{c}{4} & \multicolumn{1}{c}{5} & \multicolumn{2}{c}{6} & \multicolumn{2}{c}{7} & \multicolumn{2}{c}{8} & \multicolumn{1}{c}{9}\\ 
        \hline
         name & reference & based &$n-1$ & $w$ & \multicolumn{2}{c|}{${\cal P}$} & \multicolumn{2}{c|}{TW} & \multicolumn{2}{c|}{TW} & year\\ 
         &&  on & & & \multicolumn{2}{c|}{positions} &\multicolumn{2}{c|}{positions}&\multicolumn{2}{c|}{$w$} &  \\ \hline
         $LA$ & Langevin  & - & $20, 40, 60$ & $2\times20, 2\times30$ & rd & U[100] x U[100] & co & - & rd & U & 1993  \\ 
         & et al.  \cite{langevin1993two} && & $2\times40$ & & & & & & &  \\ \hline
         $DU$ & Dumas  & $M$\cite{langevin1993two} & $20, 40, 60,80$ & $20, 40, 60,$ & rd & U[50] x U[50] & co & - & rd & U & 1995  \\ 
         & et al. \cite{dumas1995optimal} && $100, 150, 200$ & $80, 100$ & & & & & & & \\ \hline
         $PB$ & Potvin and  & $E$\cite{solomon1987algorithms} & 3 to 44 & 25, 50, 75, & rd & U[100] x U[100] & rd 
         & U & rd & N & 1996 \\
         & Bengio \cite{potvin1996vehicle} && & 100 & cl & - & co 
         & - & &  &  \\ \hline 
         $GD$ & Gendreau & $M$\cite{dumas1995optimal} & $20, 40,60$ & $80, 100, 120, 140,$ & rd & U[50] x U[50] & co & - & fx & - & 1998 \\
         &et al. \cite{gendreau1998generalized} & & $80,100$ & $160, 180, 200$ & & & & & &  &  \\ \hline
         $PS$ & Pesant \cite{pesant1998exact} & $E$\cite{solomon1987algorithms}& 19 to 44 & $25, 50, 75,$ & rd & U[100] x U[100] & rd & U & rd & N & 1998 \\ 
         & & &  & 100 & cl & - & co & - & & &  \\ \hline
         $OT$ & Ohlmann and & $E$\cite{dumas1995optimal}& 150 and& 
         $120, 140, 160$
         & rd & U[50] x U[50] & co & - & rd & U & 2007\\
         &  Thomas \cite{ohlmann2007compressed} & & 200 & 
         & & & & & &  & \\ \hline
    \end{tabular}
    \caption{Description of TSP-TW benchmarks.
    The 6th, 7th, and 8th columns report the type of the methodology used for the generation of the positions of ${\cal P}$, the position of the time windows, and their lengths respectively. `rd' designates random, `co' constructed, `cl' clustered, and `fx' fixed. 
    In case of `rd',  we mention the used distribution: U (Uniform), N (Normal). The 3rd column indicates the original benchmark used for construction, whether it is similar in terms of methodology $M$, or as an extension $E$ of a benchmark. For PB and PS datasets, the position of the TW are randomly generated for points with random positions, and are constructed for clustered points. 
    For DU dataset, $w=100$, resp. $80$ and $60$, is not considered for $(n-1)\geq 80$, resp $(n-1)\geq 100$ and $(n-1) = 200$.
    For GD dataset, $w=80$, resp $100$, $180$ and $200$, is considered only for $(n-1)=100$, resp. $(n-1)\geq 80$, $(n-1)\leq 80$, and $(n-1)\leq 80$. 
    For OT dataset, $w=160$ is not considered for $(n-1)=200$.
    }\label{tab:ben}
\end{sidewaystable}

To generate time slots, we rely on the line segment partitioning procedures of \cite{borgos2000partitioning}, especially the method of the repulsion between the partitioning points.
Given a time horizon $h$, to produce $m$ time slots, $m-1$ points have to be placed in the interval $[0,h]$.
The repulsion method is based on randomly uniformly generating more points than needed, precisely $p\times m-1$ points, with $p>1$, and then retaining the points $p\times l$ where $l=1, .., m-1$. The points represent the bounds of the time slots, in addition to $0$ and $h$.
Figure~\ref{fig:timeslot_method} shows an example of this distribution.

\begin{figure}[h!]
    \centering
\begin{tikzpicture}[mydrawstyle/.style={draw=black, very thick}, x=1mm, y=1mm, z=1mm]
  \draw[mydrawstyle, ->](-2,30)--(66,30) node at (-6,30)[left]{$ts=1\;\;\;\;\;\;\;\;\;\;\;\;\;\;$};
  \draw[mydrawstyle](0,28)--(0,32) node[above=3]{$0$};
  \draw[mydrawstyle](10,28)--(10,32);
  \draw[mydrawstyle](20,28)--(20,32);
  \draw[mydrawstyle](30,28)--(30,32);
  \draw[mydrawstyle](40,28)--(40,32);
  \draw[mydrawstyle](50,28)--(50,32);
  \draw[mydrawstyle](60,28)--(60,32) node[above=3]{$h$};
\end{tikzpicture}

\begin{tikzpicture}[mydrawstyle/.style={draw=black, very thick}, x=1mm, y=1mm, z=1mm]
  \draw[mydrawstyle, ->](-2,30)--(66,30) node at (-6,30)[left]{$ts=2$ ($p=20$)}; 
  \draw[mydrawstyle](0,28)--(0,32);
  \draw[mydrawstyle](4.8,28)--(4.8,32);
  \draw[mydrawstyle](25.2,28)--(25.2,32);
  \draw[mydrawstyle](48.9,28)--(48.9,32);
  \draw[mydrawstyle](52.2,28)--(52.2,32);
  \draw[mydrawstyle](55.7,28)--(55.7,32);
  \draw[mydrawstyle](60,28)--(60,32);
\end{tikzpicture}

\begin{tikzpicture}[mydrawstyle/.style={draw=black, very thick}, x=1mm, y=1mm, z=1mm]
  \draw[mydrawstyle, ->](-2,30)--(66,30) node at (-6,30)[left]{$ts=3$ ($p=50$)}; 
  \draw[mydrawstyle](0,28)--(0,32);
  \draw[mydrawstyle](16.4,28)--(16.4,32);
  \draw[mydrawstyle](29.3,28)--(29.3,32);
  \draw[mydrawstyle](32.4,28)--(32.4,32);
  \draw[mydrawstyle](40.5,28)--(40.5,32);
  \draw[mydrawstyle](46.9,28)--(46.9,32);
  \draw[mydrawstyle](60,28)--(60,32);
\end{tikzpicture}

\begin{tikzpicture}[mydrawstyle/.style={draw=black, very thick}, x=1mm, y=1mm, z=1mm]
  \draw[mydrawstyle, ->](-2,30)--(66,30) node at (-6,30)[left]{$ts=4$ ($p=100$)}; 
  \draw[mydrawstyle](0,28)--(0,32);
  \draw[mydrawstyle](16.9,28)--(16.9,32);
  \draw[mydrawstyle](21.7,28)--(21.7,32);
  \draw[mydrawstyle](29.8,28)--(29.8,32);
  \draw[mydrawstyle](36.4,28)--(36.4,32);
  \draw[mydrawstyle](49.0,28)--(49,32);
  \draw[mydrawstyle](60,28)--(60,32);
\end{tikzpicture}

\begin{tikzpicture}[mydrawstyle/.style={draw=black, very thick}, x=1mm, y=1mm, z=1mm]
  \draw[mydrawstyle, ->](-2,30)--(66,30) node at (-6,30)[left]{$ts=5$ ($p=150$)}; 
  \draw[mydrawstyle](0,28)--(0,32);
  \draw[mydrawstyle](3.9,28)--(3.9,32);
  \draw[mydrawstyle](23.7,28)--(23.7,32);
  \draw[mydrawstyle](31.2,28)--(31.2,32);
  \draw[mydrawstyle](49.7,28)--(49.7,32);
  \draw[mydrawstyle](56.5,28)--(56.5,32);
  \draw[mydrawstyle](60,28)--(60,32);
\end{tikzpicture}
    \caption{An instance of time slot distribution for identical time slots ($ts=1$), and using repulsion based partitioning method \cite{borgos2000partitioning} for $p=20, 50, 100, 150$ ($ts\in [|2,5|]$).}
    \label{fig:timeslot_method}
\end{figure}
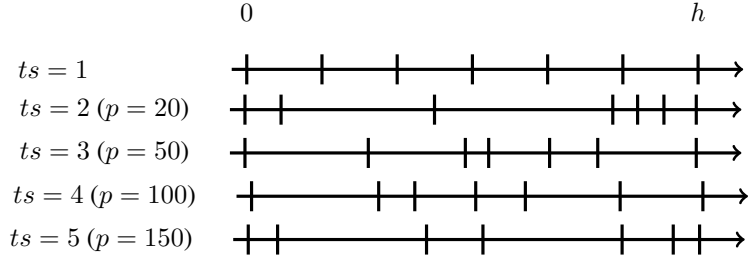

The constant $\beta^{tsp}$ is taken to be the estimates of Lei {\em et al.} \cite{lei2016solving} for values of $n \leq 90$, and to the estimates of Applegate {\em et al.} for $n \geq 100$ \cite{applegate2006traveling}, as shown in Table~\ref{tab:beta}.

The paper results are presented for DU and GD benchmarks, since they have the same base surface.
We take the time horizon to be $15$ times the diameter of the square surface, {\em i.e.} $h=1060.66$.
For each $n \in \{21, 41, 61, 81, 101\}$, five instances of DU and GD are chosen at random, thus a total of 25 initial instances. 
For each $m \in \{1, 2, 3, 4, 5, 6, 7, 8, 9, 10\}$, we generate $5$ time slot partitions $ts$: Identical time slots ($ts=1$), and time slots using the repulsion based partitioning method with $p=20, 50, 100, 150$ ($2 \leq ts \leq 5$).
For each $n \in \{21, 41, 61, 81, 101\}$, we generate $15$ distributions of the $(n-1)$ clients on $[0,h]$, $5$ uniformly (mode zero), $5$ with one mode, and $5$ with two modes. 
For one mode, we use one Normal distribution: $\mathcal{N}(\mu=h/2 , \sigma=h/4)$, and for two modes, we use a mixture of two Normal distributions: $\mathcal{N}(\mu=h/4 , \sigma=h/4)$ and $\mathcal{N}(\mu=3\: h/4 , \sigma=h/4)$.  
The uniform distribution allows to assign clients to time slots proportionally to their lengths.
The distributions with mode(s) try to replicate realistic cases of repartitions of clients.
Urban traffic flow have often two distinct peaks in terms of traffic volumes, occurring at the morning and evening.
In summary, for each initial TSP-TW instance, we generate 25 different time slot configuration, given for $1 \leq m \leq 10$,  $1 \leq ts \leq 5$, and the $15$ random distributions of clients on the time slots.
A total of $18.750$ instances are generated for benchmarking. 

The experiments are performed on an \emph{Intel(R) Core(TM) i7-8750H CPU @ 2.20GHz} processor with 32 GB RAM memory machine.

\subsection{Performance measures}

To evaluate the accuracy of the feasibility condition, we report type I (false positive - FP) and type II (false negative - FN) errors for the following null hypothesis $H_0$: the TSP-TS instance is feasible, by examining the condition (\ref{eq:pareto2}).
Type I and type II errors are essential concepts used in the interpretation of the results in statistical hypothesis testing.
In our case, the most serious error is type II error, for which the instance is predicted to be feasible, which it is not in fact. 


To examine the quality of the proposed approximations, we use a quality gap to the actual tour lengths.
This allows us to obtain a precise assessment of how close are the asymptotic approximations to the optimal length of the tours, under the assumptions of {\em Euclidean} distances and a uniform random generation of the locations of points to visit and the depot.
The quality gap of for an instance of TSP-TS is equal to 
\begin{align*}
&G_{n}^{tsp-ts} = (L_{n}^{tsp-ts} - T_{n}^{tsp-ts}) \times 100 /T_{n}^{tsp-ts},
\end{align*}
where $L_{n}^{tsp-ts}$ and $T_{n}^{tsp-ts}$ are resp. the  asymptotic approximation and the actual optimal tour length of the instance problem. 
The absolute gaps $|G_{n}^{tsp-ts}|$ are also reported.

In addition to $L_{n}^{tsp-ts}$ of proposition~\ref{prop4}, called here the distributional approximation, we compute an experimental approximation of the TSP-TS, called the sampling approximation, which takes into account the actual distribution of points to visit to the time slots. In other terms, this approximation is equal to 

\begin{equation}
L_{n}^{tsp-ts-sample} = \beta^{tsp} \sqrt{|\mathcal{R}|}(\sqrt{(1+n_1)} + \sum_{k=2}^{m} \sqrt{n_k}),
\end{equation}
where $n_k$ is the actual number of points assigned to time slot $A_k$, $1\leq k \leq m$.
The quality gap of this approximation can be noted $G_{n}^{tsp-ts-sample}$.
The actual distribution $n_k$ can be equally used to derive a  feasibility condition. An instance is feasible if
\begin{equation}
\label{feasExp}
\beta^{tsp} \sqrt{|\mathcal{R}|n_k} \leq l_k = |A_k|,\;\; \forall k \in [|1,m|].    
\end{equation}

\subsection{Quality of the approximation}
In this subsection, we examine the quality of the approximations $L_{n}^{tsp-ts}$ and $L_{n}^{tsp-ts-sample}$.
The results for the sections 7.3.1 and 7.3.2 are given for the uniform temporal distribution of clients. 
Section 7.3.3 discusses the impact of the modes of the temporal distribution.

\subsubsection{Impact of the number of points $n$}
Figure~\ref{fig:gap1} displays the absolute gaps of the proposed approximations in function of the number of clients $n$.
Knowing the actual distribution of clients on time slots leads consistently to a lower gap compared to using the distributional assumption:  $|G_{n}^{tsp-ts-sample}| \leq |G_{n}^{tsp-ts}|$.
However, the difference between both gaps becomes quite small when the number $n$ gets larger.
For $n=(81, 101)$, we have an average absolute gap of $|G_{n}^{tsp-ts}|=(5.68, 5.07\%)$ and of $|G_{n}^{tsp-ts-sample}|= (4.98, 4.06\%)$.
Both formulas are good quality approximations of the tour lengths when $n$ is large.

\begin{figure}[h]
    \centering
    \includegraphics[width=0.6\textwidth]{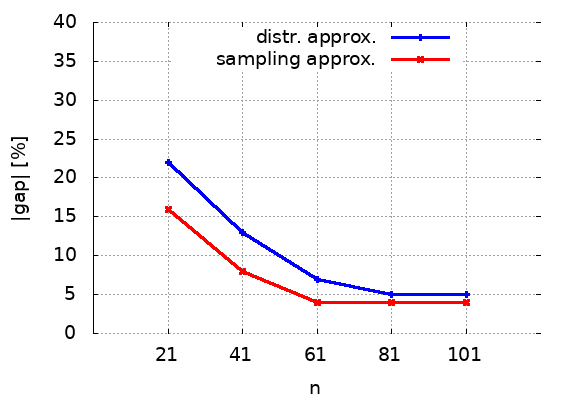}
    \caption{Average of the absolute gaps to the optimal solutions of the distr. approximation $L_{n}^{tsp-ts}$ and the sampling approximation $L_{n}^{tsp-ts-sample}$ for varying values of $n$. }
    \label{fig:gap1}
\end{figure}

\subsubsection{Impact of the time slots configuration $(m,ts)$}
\begin{figure}
    \centering
    \includegraphics[width=0.49\textwidth]{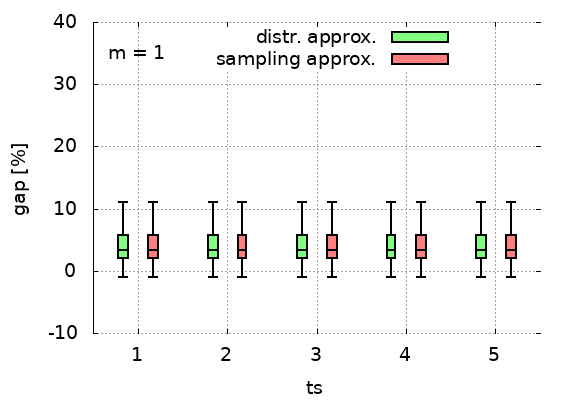}
    \includegraphics[width=0.49\textwidth]{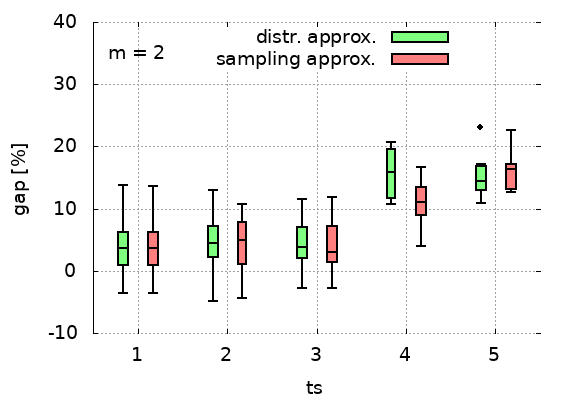}
    \includegraphics[width=0.49\textwidth]{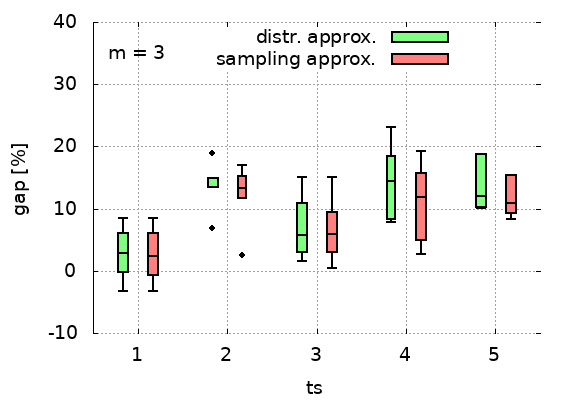}
    \includegraphics[width=0.49\textwidth]{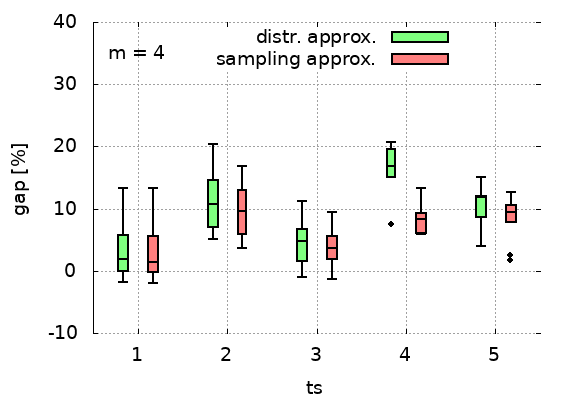}
    \includegraphics[width=0.49\textwidth]{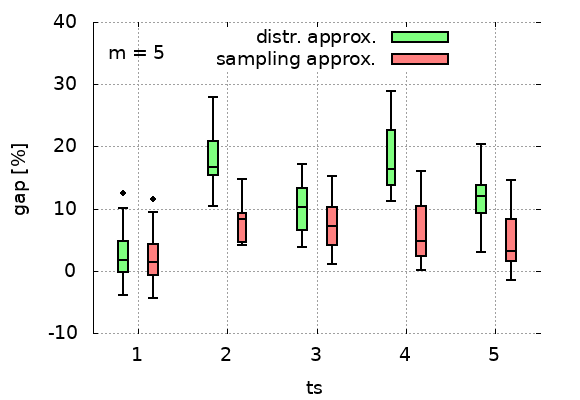}
    \includegraphics[width=0.49\textwidth]{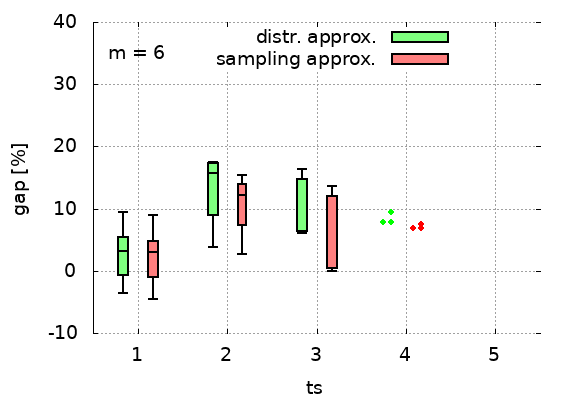}
    \caption{Distribution of gaps $G_{n}^{tsp-ts-sample}$ and $G_{n}^{tsp-ts}$, for $n\in \{81, 101\}$ and $m\in\{1, 2, 3, 4, 5, 6\}$ in function of $ts$.}
    \label{fig:gap2}
\end{figure}

The approximations' gaps for $n\in \{81, 101\}$ are shown in the Figure~\ref{fig:gap2}, when varying the time slot configuration: $1 \leq m \leq 6$ and $1 \leq ts \leq 5$.
Some observations could be made:
\begin{itemize}
\item[-] The experimental approximation $L_{n}^{tsp-ts-sample}$ is consistently better than the one based on the distributional assumption $L_{n}^{tsp-ts}$, except when $m=1$, and the cases $(m,ts=1)$ independently of the value of $m$, wherein both approximations are equivalent.
\item[-] The gaps $G_{n}^{tsp-ts-sample}$ and $G_{n}^{tsp-ts}$ are smaller for identical time slots ($ts=1$) compared to time slot configurations $ts>1$. Among time slots $ts>1$, there is no clear order in terms of the gaps. For instance, when $m=2$, the approximations' gaps get larger for increasing values of $ts$ (increasing of the repulsion parameter $p$), while for instance for $m=4$, the configurations of $ts=2$ have the largest gaps.    
\item[-] Most of the values of gaps of Figure~\ref{fig:gap2} are lower than $10\%$. The medians of the all gaps $(G_{n}^{tsp-ts}, G_{n}^{tsp-ts-sample})$ of the figure are equal to $(4.47,4.13\%)$. Some time slot configurations, such as $(m=2, ts=4)$, $(m=3, ts=2)$, $(m=3, ts=5)$, have gaps overall larger than $10\%$, but the median of the distribution of the gaps is always lower than $20\%$.
\item[-] The effect of number of time slots $m$ is negligible on the distributions of the gaps. 
\end{itemize}

\subsubsection{Impact of the distribution of points}
Figure~\ref{fig:gap3} shows the distribution of gaps in function of the temporal modes.
For identical time slots ($ts = 1$), the difference between the modes is unnoticed, except for a higher number of time slots $m$. 
For the remaining configurations of time slots ($ts>1$), the one mode, which corresponds to one normal distribution of points, tends to have higher gaps than to the two modes and the uniform distribution.
This observation is more pronounced for a large $m$. 
The reason behind this discrepancy is due to the number of time slots, which has a low allocation of clients, for which the related BHH formula

\begin{figure}[h]
    \centering
    \includegraphics[width=0.32\textwidth]{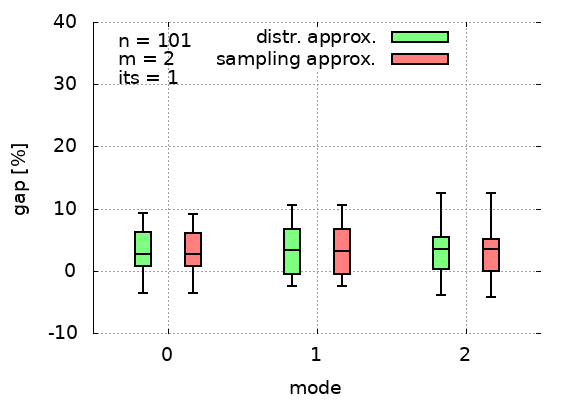}
    \includegraphics[width=0.32\textwidth]{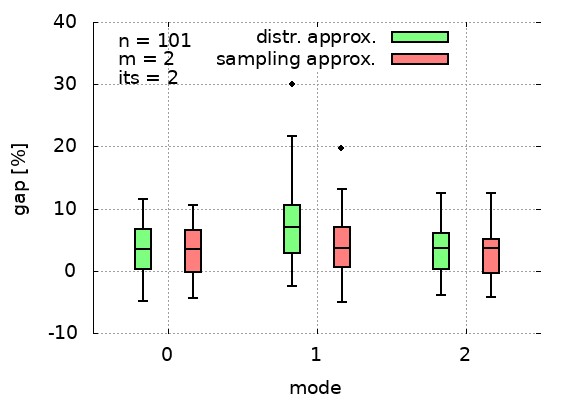}
    \includegraphics[width=0.32\textwidth]{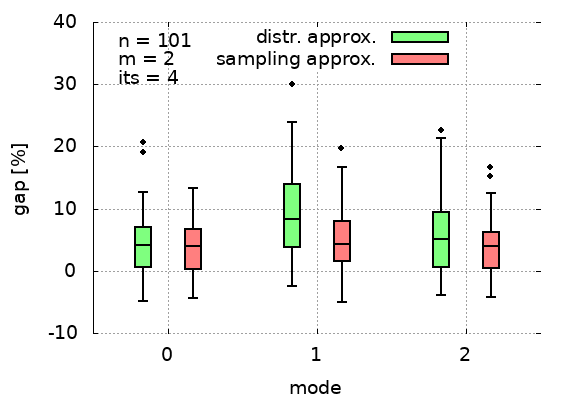}
    \includegraphics[width=0.32\textwidth]{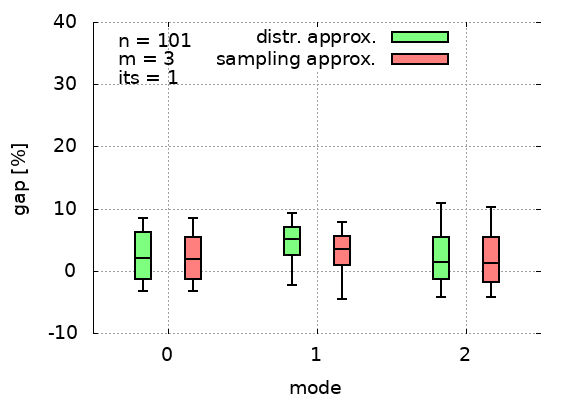}
    \includegraphics[width=0.32\textwidth]{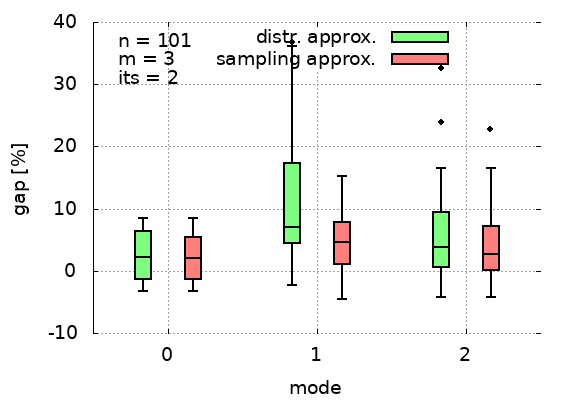}
    \includegraphics[width=0.32\textwidth]{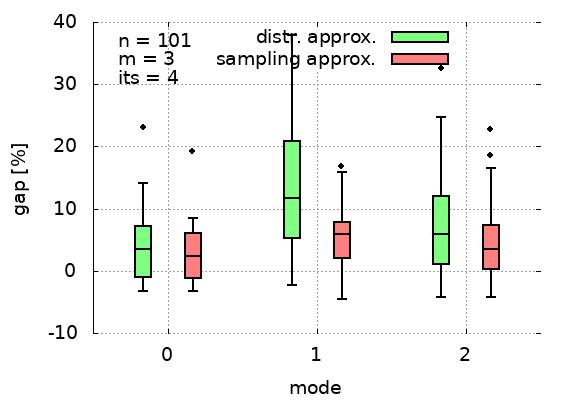}
    \includegraphics[width=0.32\textwidth]{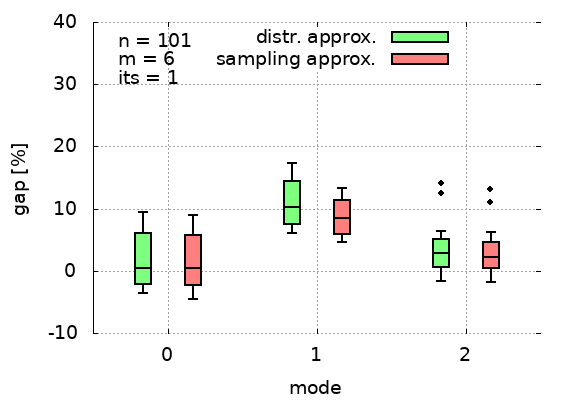}
    \includegraphics[width=0.32\textwidth]{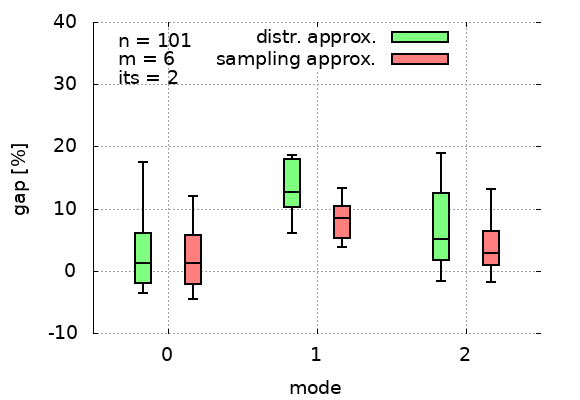}
    \includegraphics[width=0.32\textwidth]{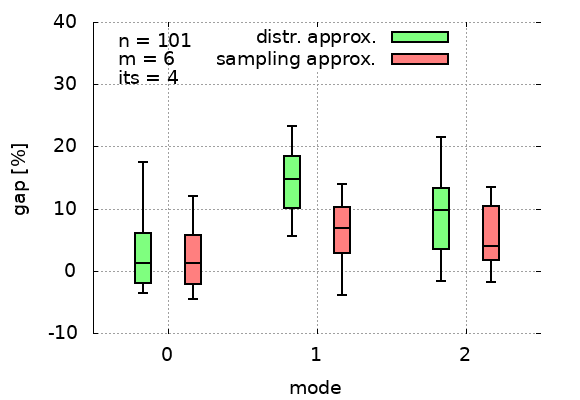}
    \caption{Distribution of gaps $G_{n}^{tsp-ts-sample}$ and $G_{n}^{tsp-ts}$, for $n = 101$, $m \in \{2, 3, 6\}$ (rows) and $ts \in \{1, 2, 4\}$ (columns), in function of the temporal distributional modes.}
    \label{fig:gap3}
\end{figure}

\noindent  behaves poorly. This number is bigger for one mode.
For instance, for $m=6$ and $n \in \{81, 101\}$, the number of time slots with less than 10 customers in our benchmark is equal to $(505, 710, 560)$ for the modes $(0,1,2)$.  

For $n= 101$, the average absolute gap of $(|G_{n}^{tsp-ts}|, |G_{n}^{tsp-ts-sample}|)$ is equal to  $(5.07, 4.06\%)$ for mode 0, $(10.38, 5.46\%)$ for mode 1,  and $(6.18, 4.49\%)$ for mode 2, 
thus an average of $(7.56, 4.76\%)$.

\subsection{Feasibility}
In this subsection, we examine the accuracy of the feasibility condition,  based on the distributional assumption, {\em i.e.} the expression (\ref{eq:pareto2}), or on the actual temporal distribution of points to visit, {\em i.e.} the expression (\ref{feasExp}). 

\subsubsection{Impact of the number of points $n$}
Figure~\ref{fig:error-n} shows the percentage of type I and type II errors in function of $n$ the number of points, for all instances with uniform temporal distribution of customers. 
For an instance to be unfeasible, only one-time slot can be sufficient,  
if the number of points of this time slot cannot be served within its length.
Relying solely on the distributional assumption, {\em i.e.} (\ref{eq:pareto2}), to discover unfeasible instances leads to high errors (the percentage of false negatives) due to this sensitivity.
Knowing the actual distribution of points on the time slot is highly beneficial in this case.
The percentage of FN for $n=61, 81, 101$ is respectively equal to $94.16, 91.02, 91.24\%$ when relying on the distributional formula, and equal to  $2.04, 2.27, 1.05\%$ if the actual distribution is known.
The BHH formula is not accurate when $n$ is low, which could explain the high percentage of the false positive error when using the actual distribution, especially since the formula is applied for each time slot.
However, this percentage decreases for large $n$, thus making the distributional information of the points to visit very useful to know beforehand.

\begin{figure}[h]
    \centering
    \includegraphics[width=0.6\textwidth]{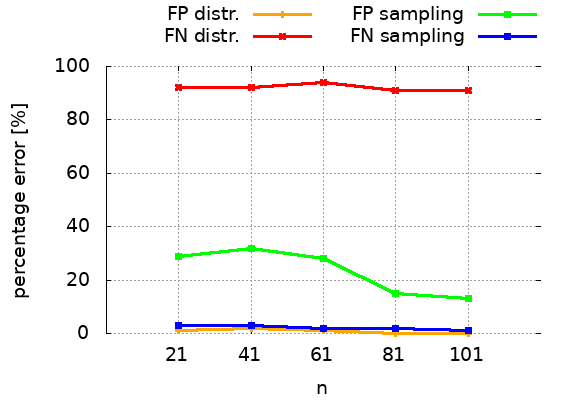}
    \caption{Averages of false positive (FP) and false negative (FN) error percentages among all instances for varying $n$,  $2 \leq m \leq 10$ and $1 \leq ts \leq 5$.}
    \label{fig:error-n}
\end{figure}

\subsubsection{Impact of the time slots configuration $(m,ts)$}
When $n$ is large, the influence of the time slots configuration $(m,ts)$ on feasibility is negligible, expect for the false positive error when relying on the actual temporal distribution.
Figure~\ref{fig:error-m} shows this impact when $n\in \{81, 101\}$.
As $m$ increases the chance to make this error becomes higher since infeasibility can be determined by only one-time slot, as said before. 
If time slots are identical in size ($ts=1$), the percentage of this error is smaller than the other repartitions $ts\in\{2, 3, 4\}$.
The percentage of all the other errors is close to zero, except for the false negative percentage when relying on the distributional assumption, which is consistently high.

\begin{figure}[h]
    \centering
    \includegraphics[width=0.49\textwidth]{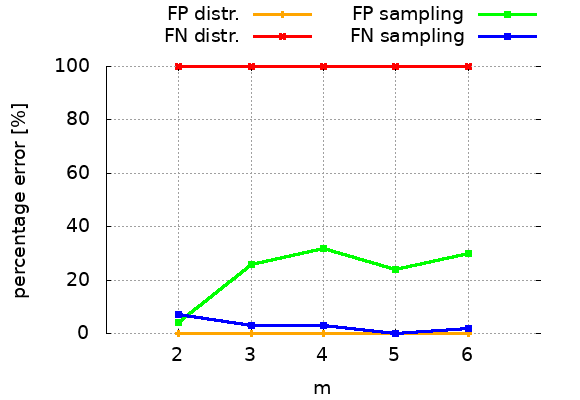}
    \includegraphics[width=0.49\textwidth]{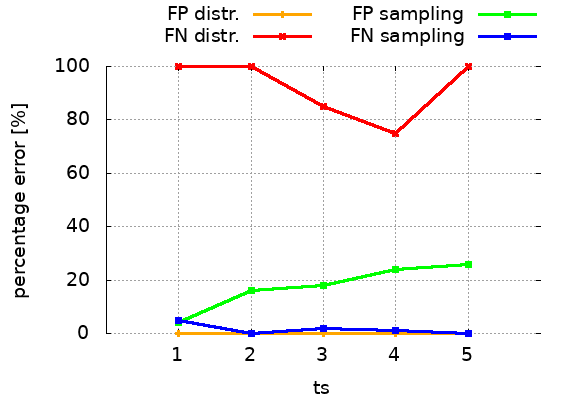}
    \caption{Averages of FP and FN errors for varying $m$ (left plot) and varying $ts$ (right plot), for $n\in \{81, 101\}$. 
    }
    \label{fig:error-m}
\end{figure}

\subsection{The TSP-TW preliminary study}

In this section, we give some preliminary results for approximating the length of TSP-TW instances in order to support the discussion of Section 4.2 with tangible data.

As said before, the upper bounds (\ref{bd:tsptw}) and (\ref{bd:tsptw1}) give poor approximations for TSP-TW in case the time windows are tight, which cover the great majority of the benchmark instances in the literature.
Table~\ref{tab:tsptw} shows the gap of the upper bound  (\ref{bd:tsptw1}) for the original 25 TSP-TW instances used in this experimentation (see Section 7.1).
The average gap is around $514\%$.

To examine the performance of the upper bound (\ref{bd:tsptw1}) on instances with large time windows and densely located in time, we generate new TSP-TW instances given the original $25$ instances. 
The length of the time windows $w$ is taken to be equal to either to a third or a quarter of the time horizon of the instance.
The location in time of the central point of the time windows is governed by a Normal distribution: $\mathcal{N}(\mu=h/2 , \sigma=h/8)$.
For each original instance and time window length $w$, $5$ new instances are created.
The average of the absolute gaps of the upper bound (\ref{bd:tsptw1}) are as follows: $(25.26, 19.34, 37.59, 46.67, 53.33\%)$ for $n=(21, 41, 61, 81, 101)$, which are much better than those of Table~\ref{tab:tsptw}. In this calculation, we did not account for time slots lower than $3\%$ of the time horizon, which can be neglected due to its size.

\begin{table}[H]
    \centering
{\footnotesize\setlength{\tabcolsep}{2pt} 
    \centering
    \begin{tabular}{c c | c c c | c | c c | }
    \hline
    \multicolumn{2}{c|}{} & \multicolumn{3}{c|}{Instance parameters}  & &  \multicolumn{2}{c|}{$L_{n-1}^{tsp-m^*ts}$} \\
    &  & &  & & & &  \\
    \# & Name & Dataset & $n-1$ & $\sqrt{|{\cal R}|}$ & Sol & apx & gap \\ \hline
1 & n20w20.004 & DU & 20 & 100 & 396.0 & 1136.9 & 187.1\\
2 & n20w120.003 & GD & 20 & 100 & 303.0 & 1121.0 & 270.0\\
3 & n20w140.004 & GD & 20 & 100 & 255.0 & 1121.0 & 339.6\\
4 & n20w200.001 & GD & 20 & 100 & 233.0 & 1071.9 & 360.0\\
5 & n20w200.004 & GD & 20 & 100 & 293.0 & 1071.9 & 265.8\\
6 & n40w20.001 & DU & 40 & 100 & 500.0 & 2047.5 & 309.5\\
7 & n40w60.002 & DU & 40 & 100 & 470.0 & 2126.0 & 352.3\\
8 & n40w80.002 & DU & 40 & 100 & 431.0 & 2126.0 & 393.3\\
9 & n40w80.004 & DU & 40 & 100 & 417.0 & 2186.7 & 424.4\\
10 & n40w200.002 & GD & 40 & 100 & 303.0 & 2094.9 & 591.4\\
11 & n60w60.002 & DU & 60 & 100 & 566.0 & 3172.7 & 460.5\\
12 & n60w60.003 & DU & 60 & 100 & 485.0 & 3143.2 & 548.1\\
13 & n60w80.005 & DU & 60 & 100 & 468.0 & 3083.3 & 558.8\\
14 & n60w140.002 & GD & 60 & 100 & 462.0 & 3052.9 & 560.8\\
15 & n60w160.001 & GD & 60 & 100 & 560.0 & 3037.6 & 442.4\\
16 & n80w20.004 & DU & 80 & 100 & 615.0 & 4034.3 & 556.0\\
17 & n80w60.005 & DU & 80 & 100 & 575.0 & 4183.3 & 627.5\\
18 & n80w80.003 & DU & 80 & 100 & 589.0 & 4049.5 & 587.5\\
19 & n80w120.001 & GD & 80 & 100 & 498.0 & 3973.2 & 697.8\\
20 & n80w160.005 & GD & 80 & 100 & 439.0 & 4168.6 & 849.6\\
21 & n100w60.004 & DU & 100 & 100 & 764.0 & 5027.8 & 558.1\\
22 & n100w120.001 & GD & 100 & 100 & 629.0 & 4955.4 & 687.8\\
23 & n100w120.003 & GD & 100 & 100 & 617.0 & 4911.5 & 696.0\\
24 & n100w120.005 & GD & 100 & 100 & 537.0 & 4984.5 & 828.2\\
25 & n100w140.002 & GD & 100 & 100 & 615.0 & 4911.5 & 698.6\\
\hline
 & Mean of gap &  & - &  & -  & -  & 514.04 \\
\hline
    \end{tabular} 
}
    \caption{Gaps of the upper bound (\ref{bd:tsptw1}). The best-known solutions of DU and GD benchmark instances are respectively given by \cite{dumas1995optimal} and \cite{bjordal2020solving}.}
    \label{tab:tsptw}
\end{table}



\section{Conclusions}
We propose an asymptotic approximation of the TSP-TS derived from the well-known TSP approximation for two variants with identical and different time slot lengths.

By means of several numerical experiments, we show that the proposed approximations provide good estimators of the tour length and the instance feasibility, even for a limited number of customers independent of the depot location.

We also show that the direct extension of our approximation does not provide a good approximation to the general case of TSP-TW. We report a gap of $4.76\%$ and $7.56\%$ of the approximation if the temporal distribution of customers on the time slots is known in advance or not\footnote{These values are obtained for $100$ customers and across different temporal modes.}. Dealing with both spatial and temporal distributions for customers under the TSP-TW assumptions requires more investigation.

The BBH approximation is given for any spatial distribution of the customers visited by the TSP. As the TSP-TS adds a temporal dimension to the TSP, we analyzed the impact of spatial and temporal distributions of customer demands through a worst-case study on the base of the maximal entropy principle. We show that our asymptotic approximation can be adapted to some worst-case hypotheses.  
 
We think that our preliminary results on the feasibility of the TSP-TS are valuable material to consider in the case of multiple vehicles, and also for the general case of overlapping time slots, which constitute the next step in our investigation. 
In addition, the results presented on the TSP-TW can be exploited to compute better approximations and bounds. The final objective of this line of research is to propose an accurate closed formula to approximate the Capacitated-VRP and the VRP-TW.




\bibliographystyle{unsrt} 
\bibliography{mybibfile}



\end{document}